\documentclass{amsart}
\usepackage{amsmath}
  \usepackage{paralist}
  \usepackage{amssymb}
 \usepackage{amsthm}
 \usepackage{esint}
 \usepackage[colorlinks=true]{hyperref}
\hypersetup{urlcolor=blue, citecolor=red}

\newtheorem{theorem}{Theorem}[section]
\newtheorem{corollary}[theorem]{Corollary}
\newtheorem{lemma}[theorem]{Lemma}

\theoremstyle{definition}

\newtheorem{remark}[theorem]{Remark}

\numberwithin{equation}{section}

\title[Lower critical Choquard equations with general nonlinearity ]
      {Existence thresholds and limit profiles of ground states for lower critical Choquard equations with general nonlinearities}

\author[Shiwang Ma and Yachen Wang]{}



\email{shiwangm@nankai.edu.cn (S.Ma), wyachen@scut.edu.cn (Y.Wang). }

\thanks{This work is supported by the National Natural Science Foundation of China (No. 11571187).\\
  {$^{\rm * }$ Corresponding author:  Shiwang Ma}
}

\begin{document}
\maketitle

\centerline{\scshape   Shiwang Ma$^{\rm 1 *}$,\; Yachen Wang$^{\rm  \rm  2}$}
\medskip
{\footnotesize
  \centerline{$^{\rm 1}$ School of Mathematical Science and LPMC, Nankai University}
 \centerline{Tianjin 300071, China}
 \centerline{$^{\rm 2}$ School of Mathematics, South China University of Technology}
  \centerline{ Guangzhou 510641, Guangdong, China }}

\bigskip

\begin{abstract}
In this paper, we study the existence, non-existence and asymptotic behavior of positive ground states for the nonlinear Choquard equation:
\begin{equation}\label{0.1}
-\Delta u+\varepsilon u=\big(I_{\alpha}\ast F(u)\big)F'(u),\quad
u\in H^1(\mathbb R^N),
\end{equation}
where $F(u)=|u|^{\frac{N+\alpha}{N}}+G(u)$ with $G(u)=\int_0^ug(s)ds$,  $N\geq3$ is an integer,  $I_{\alpha}$ is the Riesz potential of order $\alpha\in(0,N)$ and $\varepsilon>0$ is a  frequency parameter.  Under some mild subcritical growth assumptions on $g\in C([0,\infty), [0,\infty))$, we establish a  sharp threshold result  for the existence of  ground states,
and an asymptotic characterization of the ground state solutions  as $\varepsilon\to 0$. In particular,  if $g(s)\sim s^{q-1}$ as $s\to 0$ for some $q\in (\frac{N+\alpha}{N}, \frac{N+\alpha}{N-2})$,  then if $q<\frac{N+\alpha+4}{N}$, \eqref{0.1} admits a ground state for all $\varepsilon>0$,  and  if $q\ge \frac{N+\alpha+4}{N}$, then a threshold phenomena occur: there exists $\varepsilon_q>0$ such  that \eqref{0.1} has no ground state for $\varepsilon\in (0,\varepsilon_q)$ and admits a ground state for $\varepsilon>\varepsilon_q$. If $g(s)\simeq as^{q-1}$  as $s\to 0$ for some $a>0$  and $q\in (\frac{N+\alpha}{N}, \min\{\frac{N+\alpha}{N-2}, \frac{N+\alpha+4}{N}\})$, we show that
as $\varepsilon \to 0$, the ground state solutions of \eqref{0.1}, after a suitable rescaling, converges in $H^1(\mathbb R^N)$  to a particular solution of the  Hardy-Littlewood-Sobolev critical equation $u=\frac{N+\alpha}{N}(I_{\alpha}*|u|^{\frac{N+\alpha}{N}})|u|^{\frac{N+\alpha}{N}-2}u$. It turns out that the limit profiles are determined solely by the locations of $(a,q)$ in $(0,+\infty)\times (\frac{N+\alpha}{N}, \min\{\frac{N+\alpha}{N-2}, \frac{N+\alpha+4}{N}\})$. We also establish a novel sharp asymptotic characterization of such a rescaling.
\end{abstract}

\section{Introduction and Main Results}

In this paper, we consider the following Choquard equation with general nonlinearity
\begin{equation}\label{e11}
-\Delta u+\varepsilon u=(I_{\alpha}\ast F(u))F'(u),\quad
u\in H^1(\mathbb R^N),
\end{equation}
where $N\geq3$ is an integer, $\varepsilon>0$ is a frequency parameter, $I_{\alpha}$ is the Riesz potential of order $\alpha\in(0,N)$ and  is defined for every
$x\in \mathbb R^N\backslash \{0\}$ by
$$I_{\alpha}(x)=\frac{A_{\alpha}(N)}{|x|^{N-\alpha}}\ \ \mbox{and }\ \ A_{\alpha}(N)=\frac{\Gamma(\frac{N-\alpha}{2})}{\Gamma(\frac{\alpha}{2})\pi^{\frac{N}{2}}2^{\alpha}},$$ where $\Gamma$ denotes the Gamma function. We always  assume that  $F(u)=|u|^{\frac{N+\alpha}{N}}+G(u)$ with $G(u)=\int_0^ug(s)ds$ and $g\in C(\mathbb R,\mathbb R)$ being a subcritical nonlinearity, which is specified later, then equation
\eqref{e11} becomes
\begin{equation}\label{e12}
-\Delta u+\varepsilon u=(I_{\alpha}\ast(|u|^{\frac{N+\alpha}{N}}+G(u)))(\frac{N+\alpha}{N}|u|^{\frac{N+\alpha}{N}-2}u+g(u)),\quad
u\in H^1(\mathbb R^N).
\end{equation}

A prototype of such equations comes from the research of standing-wave solutions of the nonlinear Schr\"{o}dinger equation with attractive combined nonlinearities
\begin{equation}\label{e13}
\begin{aligned}
i\psi_{t}-\Delta \psi=&(I_{\alpha}*(|\psi|^{\frac{N+\alpha}{N}}+|\psi|^{q_i}+|\psi|^{q_j}))\\
\cdot &(\frac{N+\alpha}{N}|\psi|^{\frac{N+\alpha}{N}-2}\psi+q_i|\psi|^{q_i-2}\psi+q_j|\psi|^{q_j-2}\psi),
\end{aligned}
\end{equation}
in $\mathbb R^N\times\mathbb R$.
One makes the ansatz $\psi(t,x)=e^{-i\varepsilon t}u(x)$ in  \eqref{e13}, where $u: \mathbb R^N\rightarrow \mathbb C$, then  \eqref{e13} reduces to the equation \eqref{e12} with
$g(u)=q_i|u|^{q_i-2}u+q_j|u|^{q_j-2}u$. The theory about NLS with combined power nonlinearities has been first developed by T.Tao, M.Visan and X. Zhang \cite{TVZ2007}  and  received much attention during the last decades (cf. \cite{AIIKN2019, AIKN, LMZ2019, LM2019, LM2022, Ma-0, WW2022}).

A solution $u_{\varepsilon}$ of \eqref{e12} is a critical point of the corresponding action functional defined by
\begin{equation}\label{e14}
I_{\varepsilon}(u)=\frac{1}{2}\int_{\mathbb R^N}|\nabla u|^2+\frac{\varepsilon}{2}\int_{\mathbb R^N}|u|^2
-\frac{1}{2}\int_{\mathbb R^N}(I_{\alpha}*(|u|^{\frac{N+\alpha}{N}}+G(u)))
(|u|^{\frac{N+\alpha}{N}}+G(u)).
\end{equation}
The existence and local regularity properties of ground state $u_{\varepsilon}$ to \eqref{e12} are discussed  in \cite{LLT2022, LMZ2019, LM2019, MS2015}. In this paper, what we are interested in is the existence, non-existence and  the limit asymptotic
profile of the ground-states  $u_{\varepsilon}$ of the problem \eqref{e12} and the asymptotic behavior of different norms of $u_{\varepsilon}$, as
$\varepsilon \to 0$.

In \cite{MM2014},
Moroz and Muratov considered the asymptotic properties of the equation with a combined attractive-repulsive  nonlinearity in the form
\begin{equation}\label{e15}
-\Delta u+ \varepsilon u=|u|^{p-2}u- |u|^{q-2}u\ \ \mbox{in}\ \ \mathbb R^N,
\end{equation}
where $\varepsilon>0$ is a small parameter and $2<q<p<\infty$.  They pointed out that the behavior of solutions depends on whether p is less than, equal to or greater than the critical Sobolev exponent $2^*=\frac{2N}{N-2}$. Later, Liu and Moroz\cite{LM2022} extended the results in \cite{MM2014} to a class of Choquard type equation with attractive nonlocal and repulsive local interaction terms.

There are also many results about the existence and properties of ground state for Schr\"odinger equations with combined attractive nonlinearity in the form
\begin{equation}\label{e16}
-\Delta u+ \varepsilon u=|u|^{p-2}u+ |u|^{q-2}u\ \ \mbox{in}\ \ \mathbb R^N,
\end{equation}
where $2<q<p\le 2^*$. This equation has attracted the attention of many authors in the past decade.
The existence of ground states for subcritical $p<2^*$ is classical and goes back to \cite{Berestycki-1}.
In the Sobolev critical case $p=2^*$, it is well-known \cite{ASM2012, LLT2017, ZZ2012} that when $q\in (\max\{2,\frac{4}{N-2}\}, 2^*)$, 
the equation \eqref{e16} admits a ground state solution $u_\varepsilon$ for every $\varepsilon>0$. 
Recently, T. Akahori et al. \cite{Akahori-4} and Wei and Wu \cite{WW2023} proved that \eqref{e16} admits a ground state for all $N\ge 3$ and small $\varepsilon>0$, while for $N=3$ and $q\in (2,4]$, \eqref{e16} has no ground state for large $\varepsilon>0$. 
After a suitable rescaling, T. Akahori et al. \cite{AIIKN2019} established a uniform decay estimate for the ground state $u_\varepsilon$, and then proved the uniqueness and nondegeneracy in $H^1_{rad}(\mathbb R^N)$ of ground states $u_\varepsilon$ for $N\ge 5$ and large $\varepsilon>0$, and showed that for $N\ge 3$ and $q\in (\max\{2,\frac{4}{N-2}\}, 2^*)$, as $\varepsilon\to\infty$, $u_\varepsilon$ tends to a particular solution of the critical Emden--Fowler equation. 
For other related papers on the local equation \eqref{e16} we refer readers to \cite{MM2023} and the references therein.  
\smallskip

Recently, based on the existence and regularity of ground states  in \cite{LMZ2019, LM2019}, Ma and Moroz \cite{MM, MM2025}  studied the asymptotic behavior of ground states of  the following nonlinear Choquard equation
\begin{equation}\label{e17}
-\Delta u+ \varepsilon u=(I_{\alpha}*|u|^p)|u|^{p-2}u+|u|^{q-2}u\ \ \mbox{in}\ \ \mathbb R^N,
\end{equation}
where $N\geq3$ is an integer, $p\in \left[\frac{N+\alpha}{N},\frac{N+\alpha}{N-2}\right]$, $q\in \left(2,\frac{2N}{N-2}\right]$, $I_{\alpha}$ is the Riesz potential of order $\alpha\in(0,N)$ and $\varepsilon>0$ is a parameter. They established the
precisely asymptotic behaviors of positive ground state solutions as $\varepsilon \to 0$ and $\varepsilon \to \infty$ and discussed the connection to the existence and multiplicity of prescribed mass positive solutions to \eqref{e17} with the associated $L^2$ constraint condition $\int_{\mathbb R^N}|u|^2=a^2$. 
More recently in \cite{MM2024}, they extend the results in \cite{MM} to a critical Choquard equation combined with a general local perturbation term
\begin{equation}\label{e18}
-\Delta u+\varepsilon u=(I_{\alpha}*|u|^{\frac{N+\alpha}{N}})|u|^{\frac{N+\alpha}{N}-2}u+f(u)\ \ \mbox{in}\ \ \mathbb R^N.
\end{equation}
Under mild one-sided assumptions on  $f(u)$, they  prove the existence and non-existence of ground states, and establish an asymptotic characterization of the ground state solutions as $\varepsilon\to 0$ and $\varepsilon\to\infty$. 

We mention that  J. Van Schaftingen and J. Xia \cite{SX2018} proved the existence of ground states to \eqref{e18} with  $\varepsilon=1$ and subcritical local nonlinearity $f$ that satisfies Berestycki--Lions and Ambrosetti--Rabinowitz type conditions.
Recently,  X. Wang and F. Liao \cite{Wang-Liao} proved that a similar result holds true without the Ambrosetti--Rabinowitz condition.  X. Tang et al. \cite{Tang-1} studied the existence of ground states for \eqref{e18} when $f(u)$ is a pure power nonlinearity. Very recently, S. Chen et el. \cite{Chen-1} study the existence of ground state for \eqref{e18} in $\mathbb R^2$.  For more aspects about Choquard type equations, we refer the readers to the survey paper \cite{MV2017}.

Few results  have been obtained for the existence of ground states of the equation \eqref{e12} with general nonlinearity \cite{LW-2021,LLT2022}.  In particular, under some conditions on the nonlinearity  $g$  which are slightly weaker than the following general subcritical assumption 
\begin{itemize}
\item[\bf (H1)] $g\in C([0,\infty),[0,\infty))$ and $\lim_{s\to 0}g(s)/s^{\frac{\alpha}{N}}=\lim_{s\to \infty}g(s)/s^{\frac{2+\alpha}{N-2}}=0$,
\smallskip
\end{itemize}
Li et al. \cite{LLT2022}  proved  that the equation \eqref{e12} admits a radially ground state solution for any frequency $\varepsilon>0$, provided  that 
$$
\lim_{s\to 0}G(s)/s^{\frac{N+\alpha+4}{N}}=+\infty.
$$

However, to the best of our knowledge, little  is known in the case that 
$$
\limsup_{s\to 0}G(s)/s^{\frac{N+\alpha+4}{N}}<+\infty.
$$

In this case, generally speaking,   the existence of ground states depends heavily on the size of the frequency parameter $\varepsilon>0$. 

In the present paper, we will  establish an existence versus non-existence threshold result  for ground states under a  slightly stronger assumption. 
More precisely, we have the following sharp thresholds for the existence of ground states.

\begin{theorem}\label{t12} Assume $\alpha>N-4$ and  {\bf (H1)} 
 holds. 
If 
there exists  $q\in [\frac{N+\alpha+4}{N}, \frac{N+\alpha}{N-2})$ such that
$$
\limsup_{s\to 0}g(s)/s^{\frac{\alpha+4}{N}}<+\infty, \quad 
\liminf_{s\to 0}g(s)/s^{q-1}>0 
\leqno{\bf (H2)}
$$
and 
$$
\frac{N+\alpha}{N}G(s)\leq g(s)s, \quad \forall s>0,
\leqno{\bf (H3)}
$$
then there exists a positive constant $\varepsilon_{q}>0$ such that  the problem \eqref{e12} has no ground state solution for $\varepsilon\in (0,\varepsilon_{q})$ and admits a positive  ground state $u_{\varepsilon} \in H^1(\mathbb R^N)$ for $\varepsilon> \varepsilon_{q}$. Moreover,  if  $\lim_{s\to 0}g(s)/s^{\frac{\alpha+4}{N}}=0$, then the problem  \eqref{e12} with $\varepsilon=\varepsilon_q$  also admits a positive  ground state $u_{\varepsilon_q}\in H^1(\mathbb R^N)$. 
\end{theorem}

We remark that $\liminf_{s\to 0}g(s)/s^{q-1}>0$ for some $q\in [\frac{N+\alpha+4}{N}, \frac{N+\alpha}{N-2})$ is almost a necessary condition for the existence of ground states.   

If $\liminf_{s\to 0}g(s)/s^{q-1}>0$ for some $q\in (\frac{N+\alpha}{N}, \frac{N+\alpha+4}{N})$,  then $\lim_{s\to 0}g(s)/s^{\frac{\alpha+4}{N}}=\lim_{s\to 0}G(s)/s^{\frac{N+\alpha+4}{N}}=+\infty$, and then the equation \eqref{e12} has a ground state solution for any frequency $\varepsilon>0$. 
 In the present paper,  our second purpose is to study the limit asymptotic profiles of ground states as $\varepsilon\to 0$.  
 To this end, we  make the following stronger assumption 
 \smallskip
\begin{itemize}
\item[\bf (H4)] there exist $a>0$ and  $q\in (\frac{N+\alpha}{N}, \frac{N+\alpha}{N-2})$ such that
$
\lim_{s\to 0}g(s)/s^{q-1}=a.
$
\end{itemize}
 \smallskip

We show that under suitable rescaling, the limit equation of \eqref{e12} is given by the critical Hardy-Littlewood-Sobolev  equation
\begin{equation}\label{e19}
w=\frac{N+\alpha}{N}(I_{\alpha}*|w|^{\frac{N+\alpha}{N}})|w|^{\frac{N+\alpha}{N}-2}w \ \ \mbox{in}\ \ \mathbb R^N.
\end{equation}
It is well-known (see \cite{S2018}) that
\begin{equation}\label{e110}
S_{\alpha}=\inf_{u \in H^1(\mathbb R^N)\backslash\{0\}}\frac{\int_{\mathbb R^N}|u|^2}
{\left(\int_{\mathbb R^N}(I_{\alpha}*|u|^{\frac{N+\alpha}{N}})|u|^{\frac{N+\alpha}{N}}\right)^{\frac{N}{N+\alpha}}}
\end{equation}
is well-defined and achieved. Moreover,  the minimizers are  given by
\begin{equation}\label{e111}
W_1(x):=\left(\frac{A}{1+|x|^2}\right)^{\frac{N}{2}}
\end{equation}
and the family of its rescalings
\begin{equation}\label{e112}
W_{\rho}(x):=\rho^{-\frac{N}{2}}W_1(x/\rho),\ \rho>0.
\end{equation}
It is also well-known that $W_{\rho}(x)$ are radial ground states of \eqref{e19} for suitable $A>0$.

\smallskip

\begin{theorem}\label{t11} Assume  {\bf(H1)} and  {\bf (H4)} hold. If $q<\frac{N+\alpha+4}{N}$ and  $u_{\varepsilon}$ be the radial ground state of \eqref{e12}, then  as $\varepsilon \to 0$, there hold
$$u_{\varepsilon}(0)=\|u_\varepsilon\|_\infty\sim \varepsilon^{\frac{N(2+\alpha)}{\alpha(N+\alpha+4-Nq)}},\ \
\|\nabla u_{\varepsilon}\|_{2}^2\sim \varepsilon^{\frac{(N+2)(N+\alpha)-N(N-2)q}{\alpha(N+\alpha+4-Nq)}},$$
$$\|u_{\varepsilon}\|_2^2=\varepsilon^{\frac{N}{\alpha}}\left[\left(\frac{N}{N+\alpha}\right)^{\frac{N}{\alpha}}S_{\alpha}^{\frac{N+\alpha}{\alpha}}
+O\left(\varepsilon^{\frac{(2+\alpha)[Nq-(N+\alpha)]}{\alpha(N+\alpha+4-Nq)}}\right)\right],$$
$$\int_{\mathbb R^N}(I_{\alpha}*|u_{\varepsilon}|^{\frac{N+\alpha}{N}})|u_{\varepsilon}|^{\frac{N+\alpha}{N}}
=\varepsilon^{\frac{N+\alpha}{\alpha}}
\left[\left(\frac{N}{N+\alpha}S_{\alpha}\right)^{\frac{N+\alpha}{\alpha}}+O\left(\varepsilon^\frac{(2+\alpha)[Nq-(N+\alpha)]}{\alpha(N+\alpha+4-Nq)}\right)\right].$$
Moreover, there exists $\xi_{\varepsilon} \in (0,\infty)$ verifying
$\xi_{\varepsilon}\sim \varepsilon^{-\frac{N(q-1)+\alpha}{\alpha(4+N+\alpha-Nq)}}$
such that for small $\varepsilon>0$, the rescaled ground states
$w_{\varepsilon}(x)=\varepsilon^{-\frac{N}{2\alpha}}\xi_{\varepsilon}^{\frac{N}{2}}u_{\varepsilon}(\xi_{\varepsilon}x)$
satisfies
$$\|\nabla w_{\varepsilon}\|_2^2\thicksim \|w_{\varepsilon}\|_2^2\thicksim
\int_{\mathbb R^N}(I_{\alpha}*|w_{\varepsilon}|^{\frac{N+\alpha}{N}})|w_{\varepsilon}|^{\frac{N+\alpha}{N}}\thicksim
\int_{\mathbb R^N}(I_{\alpha}*|w_{\varepsilon}|^{\frac{N+\alpha}{N}})|w_{\varepsilon}|^{q}\thicksim1,$$
and as $\varepsilon \to 0$, $w_{\varepsilon}$ converges to $W_{\rho_0}$  in $H^1(\mathbb R^N)$,  where 
$$\rho_0=\left(\frac{2q\int_{\mathbb R^N}|\nabla W_1|^2}
{a[Nq-(N+\alpha)]\int_{\mathbb R^N}(I_{\alpha}*|W_1|^{\frac{N+\alpha}{N}})|W_1|^{q}}\right)^{\frac{2}{N+\alpha+4-Nq}}.$$
Furthermore,  as $\varepsilon\to 0$, the least energy  $I_{\varepsilon}(u_{\varepsilon})$ of the ground state satisfies
$$\frac{\alpha}{2N}\left(\frac{N}{N+\alpha}S_{\alpha}\right)^{\frac{N+\alpha}{\alpha}}-\varepsilon^{-\frac{N+\alpha}{\alpha}}I_{\varepsilon}(u_{\varepsilon})\sim 
\varepsilon^{\frac{(2+\alpha)[Nq-(N+\alpha)]}{\alpha(N+\alpha+4-Nq)}}.
$$
\end{theorem}
\smallskip 

\smallskip


Surprisingly, in Theorem \ref{t11}, we establish that if  $q\in (\frac{N+\alpha}{N},\frac{N+\alpha+4}{N})$, then as $\varepsilon \to 0$,  the rescaled family of ground states
$w_\varepsilon$  converges  to the extremal function  $W_{\rho_0}$ in $H^1(\mathbb R^N)$, rather than in $L^2(\mathbb  R^N)$, even if the limit equation is well--posed in $L^2(\mathbb R^N)$ only. We divide the  associated $(\alpha, q)$ region $\left\{(\alpha, q): \ \frac{N+\alpha}{N}<q<\frac{N+\alpha}{N-2},  \ 0<\alpha<N \right\}$ into two parts:  
$$
\mathcal A_N=\left\{(\alpha,q): \ \frac{N+\alpha+4}{N}\le q<\frac{N+\alpha}{N-2}\right\}.
$$

$$
\mathcal B_N=\left\{ (\alpha, q): \ \frac{N+\alpha}{N}<q<\min\{\frac{N+\alpha}{N-2},\frac{N+\alpha+4}{N}\} \right\},
$$
which are depicted in Fig. 1.  Then Theorem \ref{t12} and Theorem \ref{t11} apply to $\mathcal A_N$ and $\mathcal B_N$, respectively. 
\vspace{ 6cm}

\begin{figure}[h]
	\caption{The regimes in the $(\alpha,q)$ plane where Theorem \ref{t12} and Theorem \ref{t11} are applicable.}
\label{table1}
\end{figure}

We highlight some new findings in this paper. 
Firstly, we find a new  threshold exponent $q_0=\frac{N+\alpha+4}{N}$ in the case that $\alpha>N-4$. This new exponent is larger than the mass critical exponent $\frac{N+\alpha+2}{N}$ for the Choquard equations. When the subcritical perturbation $g(s)$ satisfies $g(s)\sim s^{q-1}$ as $s\to 0$, then the existence and non-existence of ground states are determined by the location of the exponent $q$ in $(\frac{N+\alpha}{N}, \frac{N+\alpha}{N-2})$. The exponent $q_0$ is critical in the sense that if $q\in (\frac{N+\alpha}{N}, \frac{N+\alpha+4}{N})$, then \eqref{e12} admits a ground state for any frequency $\varepsilon>0$, while if $q\in [\frac{N+\alpha+4}{N},\frac{N+\alpha}{N-2})$, then the existence of ground state depends heavily on the size of  frequency $\varepsilon>0$ and a threshold phenomena occur: there exists $\varepsilon_q>0$ such that \eqref{e12} has no ground state for $\varepsilon \in (0,\varepsilon_q)$ and admits a ground state for $\varepsilon \ge \varepsilon_q$ ($\varepsilon>\varepsilon_q$ in the case $q=\frac{N+\alpha+4}{N}$).  Similar threshold phenomena are known for critical combined Schr\"odinger equation \cite{WW2023} and lower critical Choquard equation  combined with a general local perturbation term \cite{MM2024}, but the correspoding critical exponents are totally different. In the former case similar critical exponent occur only in $N=3$ and $q_0$ equals to 4, and in  the latter case the corresponding  critical exponent $q_0$ equals to the mass critical exponent $2+\frac{4}{N}$ for the Schr\"odinger equation. 

Secondly, 
if $g(s)\simeq as^{q-1}$ as $s\to 0$ for some $q\in (\frac{N+\alpha}{N},\frac{N+\alpha+4}{N})$ and some $a>0$, 
then as $\varepsilon\to 0$, the limit asymptotic profiles of ground states are determined solely by the location of $(a, q)$ in $(0,\infty)\times (\frac{N+\alpha}{N}, \frac{N+\alpha+4}{N})$, which is somewhat a  striking fact. In order to incorporate lower-order terms caused by $g(u)$, we first use  a suitable rescaling  to transform
the equation \eqref{e12}  into the equation
$$
-\varepsilon^{\sigma}\Delta w+w=\left(I_{\alpha}*(|w|^{\frac{N+\alpha}{N}} +\varepsilon_1^{-1}G(\varepsilon_2w))\right)
\left(\frac{N+\alpha}{N}|w|^{\frac{N+\alpha}{N}-2}w
+\varepsilon_1^{-1}\varepsilon_2g(\varepsilon_2w)\right),
$$
where
$\sigma:=\frac{(2+\alpha)[Nq-(N+\alpha)]}{\alpha(N+\alpha+4-Nq)}>0$, $\varepsilon_1, \varepsilon_2=o_\varepsilon(1)$, and  satisfies $\varepsilon^\sigma=\varepsilon_1^{-1}\varepsilon_2^{q}$. 
Then we show that the rescaled family of ground state solutions $\{w_\varepsilon\}$ converges to a limit $w_0\neq 0$ in $L^2(\mathbb R^N)$. 
Generally, the convergence of $w_\varepsilon$ to $w_0$ in $D^{1,2}(\mathbb R^N)$ is not expected. 
To show the existence of a rescaling which ensures convergence in $H^1(\mathbb R^N)$ (rather than only in $L^2(\mathbb R^N)$) to a particular solution of a limit equation, we employ the identity
\begin{equation}\label{e113}
\|\nabla w_{\varepsilon}\|_2^2=\frac{a[Nq-(N+\alpha)]}{2q}
\int_{\mathbb R^N}(I_{\alpha}*|w_{\varepsilon}|^{\frac{N+\alpha}{N}})|w_{\varepsilon}|^{q}+\mathcal{R_\varepsilon}.
\end{equation}
If $ \mathcal R_\varepsilon=o_\varepsilon(1)$, then the identity \eqref{e113}  allows us to adjust $w_\varepsilon$ by introducing a new $L^2$-norm preserving rescaling $\tilde w_\varepsilon(x)=\rho_\varepsilon^{N/2}w_\varepsilon(\rho_\varepsilon x)$ with a suitable $\rho_\varepsilon>0$, in such a way that the weak limit $\tilde w_0$ of $\tilde w_\varepsilon$ satisfies $\lim_{\varepsilon\to 0}\|\nabla \tilde w_\varepsilon\|_2=\|\nabla \tilde w_0\|_2$. 
This implies that $\tilde w_\varepsilon\to \tilde w_0$ in $D^{1,2}(\mathbb R^N)$, and hence in $H^1(\mathbb R^N)$.
We remark that proving the boundedness of $\{w_\varepsilon\}$  in $H^1(\mathbb R^N)$ is a main component in adopting this technique. However,  since $\mathcal R_\varepsilon\not=0$ due to the presence of the cross terms, the proof of this issue becomes much more involved  and some new techniques are required. In the present paper we first show that $\|w_\varepsilon\|_2^2$ and $\varepsilon^\sigma \|\nabla w_\varepsilon\|_2^2$ are bounded. By Hardy-Littlewood-Sobolev and Sobolev inequalities, we also show that  $\mathcal R_\varepsilon\lesssim \sum_{i=1}^k\varepsilon^{\tau_i\sigma}\|\nabla w_\varepsilon\|_2^{\beta_i}$, with $\tau_i>0$ and $\beta_i>0$.  Therefore, it follows from \eqref{e113} that 
\begin{equation}\label{e114}
\|\nabla w_\varepsilon\|_2^2\lesssim \|\nabla w_\varepsilon\|_2^{\frac{Nq-(N+\alpha)}{2}}+ \sum_{i=1}^k\varepsilon^{\tau_i\sigma}\|\nabla w_\varepsilon\|_2^{\beta_i}.
\end{equation}
Since $\beta_i>2$ for some $i\in \{1,2,\cdots, k\}$, we can not conclude the boundedness of $\|\nabla w_\varepsilon\|_2$ from \eqref{e114}.  To overcome this difficulty, we find some numbers $\tilde \tau_i\in [0,\tau_i)$ such that $0\le \beta_i-2\tilde \tau_i\le 2$,   and then by the boundedness of $\varepsilon^\sigma \|\nabla w_\varepsilon\|_2^2$ we obtain
$\mathcal R_\varepsilon\lesssim \sum_{i=1}^k \varepsilon^{(\tau_i-\tilde \tau_i)\sigma}\|\nabla w_\varepsilon\|_2^{\beta_i-2\tilde \tau_i}$. This estimate of  the remainder term $\mathcal R_\varepsilon$ allows us  to obtain $\|\nabla w_\varepsilon\|_2^2\lesssim \|\nabla w_\varepsilon\|_2^{\frac{Nq-(N+\alpha)}{2}}+ \sum_{i=1}^k\varepsilon^{(\tau_i-\tilde\tau_i)\sigma}\|\nabla w_\varepsilon\|_2^{\beta_i-2\tilde \tau_i}$  by using  \eqref{e113}, which  together with the fact that  $\frac{Nq-(N+\alpha)}{2}<2$ implies the boundedness of $\|\nabla w_\varepsilon\|_2$. See Lemma \ref{lem42} below for details. We remark that  a similar technique is also used in the proof of Lemma \ref{lem43}.
This reduction technique seems to be new and should be applicable in other related context.  

This paper is organized as follows. In section 2, we present some preliminary
lemmas. Then we prove our main results, Theorem \ref{t12} and Theorem \ref{t11}, in section 3 and section 4, respectively.  Finally, in the last section, we discuss briefly the asymptotic behavior of ground states as $\varepsilon\to \infty$. 

Throughout this paper, we assume $N\geq 3$. 
$B_r$ denotes the ball in $\mathbb R^N$ with radius $r>0$ and centered at the origin,  $|B_r|$ and $B_r^c$ denote its Lebesgue measure  and its complement in $\mathbb R^N$, respectively.

$L^p(\mathbb{R}^N)$ with $1\leq p<\infty$ is the Lebesgue space
with the norm
$\|u\|_p=\left(\int_{\mathbb{R}^N}|u|^p\right)^{1/p}$.

$ H^1(\mathbb{R}^N)$ is the usual Sobolev space with norm
$\|u\|_{H^1(\mathbb{R}^N)}=\left(\int_{\mathbb{R}^N}|\nabla
u|^2+|u|^2\right)^{1/2}$. 

$ D^{1,2}(\mathbb{R}^N)=\{u\in
L^{2^*}(\mathbb{R}^N): |\nabla u|\in
L^2(\mathbb{R}^N)\}$. 

\noindent
We denote by $c,C$, various positive constants.
For any $s>0$ and two nonnegative functional $f(s)$ and $g(s)$, we write 
\begin{itemize}
\item[(i)] $f(s)\lesssim g(s)$ or $g(s)\gtrsim f(s)$ if there exists a positive constant $C$ independent of $s$ such that
$f(s)\leq Cg(s)$.
\item[(ii)] $f(s) \backsim g(s)$ if $f(s)\lesssim g(s)$ and $f(s)\gtrsim g(s)$.
\end{itemize}
If $|f(s)|\lesssim |g(s)|$, we write $f(s)=O((g(s))$.  Finally, if $\lim f(s)/g(s)=1$, then we write $f(s)\simeq g(s)$.



\smallskip


\section{Preliminary}

In order to prove our results, we give some lemmas, which are useful for the subsequent proof. First, we give the well known Hardy-Littlewood-Sobolev inequality, which can be found in \cite{LL2001}.
\begin{lemma}\label{lem21}
Let $p$, $r>1$ and $0<\alpha<N$ with $1/p+(N-\alpha)/N+1/r=2$. Let $u \in L^{p}(\mathbb R^N)$ and $v \in L^{r}(\mathbb R^N)$. Then there exists a sharp constant
$C(N,\alpha,p)$, independent of $u$ and $v$, such that
$$\left|\int_{\mathbb R^N}\int_{\mathbb R^N}\frac{u(x)v(y)}{|x-y|^{N-\alpha}}\right|\leq C(N,\alpha,p)\|u\|_{p}\|v\|_{r}.$$
If $p=r=\frac{2N}{N+\alpha}$, then
$$C(N,\alpha,p)=C_{\alpha}(N)=\pi^{\frac{N-\alpha}{2}}
\frac{\Gamma(\frac{\alpha}{2})}{\Gamma(\frac{N+\alpha}{2})}\left\{\frac{\Gamma(\frac{N}{2})}{\Gamma(N)}\right\}^{-\frac{\alpha}{N}}.$$
\end{lemma}

\begin{remark}
By the Hardy-Littlewood-Sobolev inequality, for any $v\in L^{s}(\mathbb R^N)$ with
$s \in (1,\frac{N}{\alpha})$, $I_{\alpha}*v \in L^{\frac{Ns}{N-\alpha s}}(\mathbb R^N)$ and
$$\|I_{\alpha}*v\|_{\frac{Ns}{N-\alpha s}}\leq A_{\alpha}(N)C(N,\alpha,s)\|v\|_s.$$
\end{remark}

The Gagliardo-Nirenberg inequality can be seen in \cite{Weinstein}.
\begin{lemma}\label{lem23}
Let $N \geq 1$, $2< s <2^*$ and $\gamma_s =\frac{N(s-2)}{2s} $, then the following sharp inequality
$$\| u \|_s \leq C_{N,s}\| u\|_2^{1-\gamma_s} \| \nabla u\|_2^{\gamma_s}    $$
holds for any $u \in H^1(\mathbb R^N)$, where the constant $C_{N,s}=\left( \frac{s}{  2\| Q_s \|_2^{s-2}  }  \right)^{\frac{1}{s}}$
and $Q_s$ is the unique positive solution of
\begin{equation}\label{equ:Q sovles equ}
  -\frac{N(s-2)}{4} \Delta Q_s + \frac{2N-s(N-2)}{4} Q_s=|Q_s|^{s-2}Q_s, \ \ x\in \mathbb{R}^N.
\end{equation}
\end{lemma}

Arguing as  in \cite[Theorem 2.1]{LM2019} it is easy to show that any weak solution of \eqref{e12} in
$H^1(\mathbb{R}^N)$ has additional regularity properties, which
allows us to establish the Poho\v{z}aev identity for all finite
energy solutions.

\smallskip

\begin{lemma}\label{lem24}
 Assume {\bf (H1)} holds.  If
$u\in H^1(\mathbb{R}^N)$ is a solution of \eqref{e12}, then $u\in
W_{\mathrm{loc}}^{2,r}(\mathbb{R}^N)$ for every $r>1$. Moreover, $u$
satisfies the Poho\v{z}aev identity
\begin{equation}\label{e22}
\begin{array}{rcl}
P_{\varepsilon}^0(u):&=&\frac{N-2}{2}\int_{\mathbb{R}^N}|\nabla
u|^2+\frac{N\varepsilon}{2}\int_{\mathbb{R}^N}|u|^2\\
&\mbox{}&
-\frac{N+\alpha}{2}\int_{\mathbb R^N}(I_{\alpha}*(|u|^{\frac{N+\alpha}{N}}
+G(u)))(|u|^{\frac{N+\alpha}{N}}+G(u))dx=0.
\end{array}
\end{equation}
\end{lemma}

It is well known that any weak solution of  \eqref{e12} corresponds to a critical point of the action  functionals $I_{\varepsilon}$ defined in \eqref{e14},
which is well defined and is of $C^1$ in $H^1(\mathbb R^N)$.  A nontrivial solution $u_\varepsilon\in H^1(\mathbb R^N)$  is called a ground state if
\begin{equation}\label{e23}
I_\varepsilon(u_{\varepsilon})=m_{\varepsilon}:=\inf\{ I_{\varepsilon}(u): \ u\in H^1(\mathbb R^N)\setminus \{0\} \ {\rm and}  \  I'_{\varepsilon}(u)=0\}.
\end{equation}
In \cite{LMZ2019, LM2019} (see also the proof of the main results in \cite{LM2019}),  it has been shown that
\begin{equation}\label{e24}
m_{\varepsilon}=\inf_{u\in \mathcal N_{\varepsilon}^0}I_{\varepsilon}(u)=\inf_{u\in \mathcal P_{\varepsilon}^0}I_{\varepsilon}(u),
\end{equation}
where $\mathcal N_{\varepsilon}^0$ and $\mathcal P_{\varepsilon}^0$ are the correspoding Nehari  and Poho\v{z}aev manifolds defined by
$$
\mathcal N_{\varepsilon}^0:=\left\{ u\in H^1(\mathbb R^N)\setminus\{0\}  \ \left | \ \begin{array}{rl}
&\int_{\mathbb R^N}|\nabla u|^2+\varepsilon |u|^2\\
&=\int_{\mathbb R^N}(I_{\alpha}*(|u|^{\frac{N+\alpha}{N}}+G(u))) (\frac{N+\alpha}{N}|u|^{\frac{N+\alpha}{N}}+g(u)u)dx \end{array}\right.\right\}
$$
and
$$
\mathcal P_{\varepsilon}^0:=\left\{ u\in H^1(\mathbb R^N)\setminus\{0\}  \ \left | \ P_{\varepsilon}^0(u)=0 \right. \right\},
$$
respectively.

In the following, we describe the following minimax characterizations for the least energy $m_{\varepsilon}$.
\begin{lemma}\label{lem25}
Let
$$u_{t}(x)=
\left\{
\begin{aligned}
&u \left(\frac{x}{t}\right),
\ & if& \ t>0,\\
&0,
\ & if&\ t=0,\\
\end{aligned}
\right.$$
then
$$m_{\varepsilon}=\inf_{u\in H^{1}(\mathbb R^N)\backslash\{0\}}\sup_{t\geq 0}I_{\varepsilon}(tu)=\inf_{u\in H^{1}(\mathbb R^N)\backslash\{0\}}\sup_{t\geq 0}I_{\varepsilon}(u_t).$$
In particular, if $u_\varepsilon$ is a ground state, then we have $m_{\varepsilon}=I_{\varepsilon}(u_{\varepsilon})=\sup_{t>0}I_{\varepsilon}(tu_{\varepsilon})=\sup_{t>0}I_{\varepsilon}((u_{\varepsilon})_t)$.
\end{lemma}
\begin{proof}
The proof is standard, so we omit it here, and we refer readers to \cite{JJ2002,LM2019}.
\end{proof}

\section{Proof of Theorem \ref{t12}}
In this section, we consider the existence and nonexistence of ground states for the equation \eqref{e12} in case $\alpha\geq N-4$. 
In this case,  we have $\frac{N+\alpha+4}{N}\le \frac{N+\alpha}{N-2}$.
 We make  a canonical rescaling:
\begin{equation}\label{e31}
v(x)=\varepsilon^{-\frac{N(2+\alpha)}{4\alpha}}u(\varepsilon^{-\frac{1}{2}} x),
\end{equation}
then \eqref{e12} transforms into the equation
\begin{equation}\label{e32}
\begin{aligned}
-\Delta v+ v=&\left(I_{\alpha}*(|v|^{\frac{N+\alpha}{N}}+\varepsilon^{-\frac{(N+\alpha)(2+\alpha)}{4\alpha}}G(\varepsilon^{\frac{N(2+\alpha)}{4\alpha}}v))\right)\\
\cdot &\left(\frac{N+\alpha}{N}|v|^{\frac{N+\alpha}{N}-2}v+\varepsilon^{-\frac{2+\alpha}{4}}g(\varepsilon^{\frac{N(2+\alpha)}{4\alpha}}v)\right).\\
\end{aligned}
\end{equation}

The formal limit equation of \eqref{e32} as $\varepsilon\to 0$ is given by
\begin{equation}\label{e33}
-\Delta v+ v=\frac{N+\alpha}{N}(I_{\alpha}*|v|^{\frac{N+\alpha}{N}})|v|^{\frac{N+\alpha}{N}-2}v.
\end{equation}

We denote the Nehari manifolds for \eqref{e32} and
\eqref{e33} as follows:
$$\tilde{ \mathcal{N}}_{\varepsilon}:=\{v \in H^1(\mathbb R^N)\backslash\{0\}\mid \tilde  N_{\varepsilon}(v)=0\},$$
and
$$\tilde{\mathcal{N}}_{0}:=\{v \in H^1(\mathbb R^N)\backslash\{0\}\mid
\int_{\mathbb R^N}|\nabla v|^2+\int_{\mathbb R^N}|v|^2=\frac{N+\alpha}{N}\int_{\mathbb R^N}(I_{\alpha}*|v|^{\frac{N+\alpha}{N}})|v|^{\frac{N+\alpha}{N}}\},$$
where
$$\begin{aligned}
\tilde N_{\varepsilon}(v)=\int_{\mathbb R^N}|\nabla v|^2+\int_{\mathbb R^N}|v|^2
-\int_{\mathbb R^N}&(I_{\alpha}*(|v|^{\frac{N+\alpha}{N}}
+\varepsilon^{-\frac{(N+\alpha)(2+\alpha)}{4\alpha}}G(\varepsilon^{\frac{N(2+\alpha)}{4\alpha}}v)))\\
\cdot&(\frac{N+\alpha}{N}|v|^{\frac{N+\alpha}{N}}+\varepsilon^{-\frac{2+\alpha}{4}}g(\varepsilon^{\frac{N(2+\alpha)}{4\alpha}}v)v)dx.\\
\end{aligned}$$
The corresponding energy functional of \eqref{e32} is given by
\begin{equation}\label{e34}
\begin{aligned}
\tilde J_{\varepsilon}(v)=\frac{1}{2}\int_{\mathbb R^N}|\nabla v|^2+ |v|^2
-\frac{1}{2}\int_{\mathbb R^N}&(I_{\alpha}*(|v|^{\frac{N+\alpha}{N}}
+\varepsilon^{-\frac{(N+\alpha)(2+\alpha)}{4\alpha}}G(\varepsilon^{\frac{N(2+\alpha)}{4\alpha}}v)))\\
\cdot&(|v|^{\frac{N+\alpha}{N}}+\varepsilon^{-\frac{(N+\alpha)(2+\alpha)}{4\alpha}}G(\varepsilon^{\frac{N(2+\alpha)}{4\alpha}}v))dx,\\
\end{aligned}
\end{equation}
and the energy functional of the limit equation \eqref{e33} is
\begin{equation}\label{e35}
\tilde J_{0}(v)=\frac{1}{2}\int_{\mathbb R^N}|\nabla v|^2+|v|^2
-\frac{1}{2}\int_{\mathbb R^N}(I_{\alpha}*|v|^{\frac{N+\alpha}{N}})|v|^{\frac{N+\alpha}{N}}.
\end{equation}
It is easy to see that
$$\tilde m_{\varepsilon}:=\inf_{v\in \tilde{\mathcal{N}}_{\varepsilon}}\tilde J_{\varepsilon}(v)\ \ \mbox{and}\ \  \tilde m_{0}:=\inf_{v \in \tilde{ \mathcal{N}}_{0}}\tilde J_{0}(v)$$
are well defined and positive. Moreover,  it is not hard to show that    
$$
\tilde m_{\varepsilon}=\inf_{u\in H^1(\mathbb R^N)\setminus\{0\}}\sup_{t\ge 0}\tilde{J}_{\varepsilon}(tu)
=\inf_{u\in H^1(\mathbb R^N)\setminus\{0\}}\sup_{t\ge 0}\tilde{J}_{\varepsilon}(u_t).
$$
Similar min-max characterization holds true for $\tilde m_0$. Let  $S_\alpha$ be defined by \eqref{e110}, then it is standard to show that  
\begin{equation}\label{e36}
S_{\alpha}=\inf_{v \in H^1(\mathbb R^N)\backslash\{0\}}\frac{\int_{\mathbb R^N}|\nabla v|^2+\int_{\mathbb R^N}|v|^2}
{\left(\int_{\mathbb R^N}(I_{\alpha}*|v|^{\frac{N+\alpha}{N}})|v|^{\frac{N+\alpha}{N}}\right)^{\frac{N}{N+\alpha}}}
\end{equation}
and  $\tilde m_0=\frac{\alpha}{2N}(\frac{N}{N+\alpha}S_\alpha)^{\frac{N+\alpha}{\alpha}}$  is not attained.  By the min-max descriptions of $\tilde m_\varepsilon$ and $\tilde  m_0$,  it is easy to show that $\tilde m_\varepsilon\le \tilde m_0=\frac{\alpha}{2N}(\frac{N}{N+\alpha}S_\alpha)^{\frac{N+\alpha}{\alpha}}$ for all $\varepsilon>0$.

\begin{proof}[Proof of Theorem \ref{t12}]   
By the assumption {\bf (H3)}, for any fixed $v\in H^1(\mathbb R^N)$,  we have 
$$
\begin{array}{rcl}
\frac{d}{d\varepsilon}\tilde J_\varepsilon(v)&=&\frac{N(2+\alpha)}{4\alpha}\varepsilon^{-\frac{(N+\alpha)(2+\alpha)+4\alpha}{4\alpha}}\int_{\mathbb R^N}(I_\alpha\ast (|v|^{\frac{N+\alpha}{N}}+\varepsilon^{-\frac{(N+\alpha)(2+\alpha)}{4\alpha}}G(\varepsilon^{\frac{N(2+\alpha)}{4\alpha}}v)))\\
&\mbox{} & \qquad\qquad \qquad \qquad  \qquad \cdot (\frac{N+\alpha}{N}G(\varepsilon^{\frac{N(2+\alpha)}{4\alpha}}v)-g(\varepsilon^{\frac{N(2+\alpha)}{4\alpha}}v)\varepsilon^{\frac{N(2+\alpha)}{4\alpha}}v)dx\\
&\le & 0.
\end{array}
$$
Therefore,  $\tilde J_\varepsilon(v)$ is nonincreasing with respect to $\varepsilon>0$ for any fixed $v\in H^1(\mathbb R^N)$.  For any $\varepsilon_1<\varepsilon_2$ and any $v\in \tilde{\mathcal{N}}_{\varepsilon_1}$, there exists $t>0$ such that $tv\in \tilde{\mathcal{N}}_{\varepsilon_2}$. Therefore, we obtain
$$
\tilde m_{\varepsilon_2}\le \tilde J_{\varepsilon_2}(tv)\le \tilde J_{\varepsilon_1}(tv)\le \tilde J_{\varepsilon_1}(v),
$$
from which it follows that $\tilde m_{\varepsilon_2}\le \tilde m_{\varepsilon_1}$. That is,  $\tilde m_\varepsilon$ is nonincreasing  with respect to $\varepsilon$.    On the other hand, if $\tilde m_\varepsilon<\frac{\alpha}{2N}(\frac{N}{N+\alpha}S_\alpha)^{\frac{N+\alpha}{\alpha}}$ for some $\varepsilon>0$, arguing as in \cite{LLT2022} , we can show that \eqref{e12} 
has a ground state solution.

We claim that \eqref{e12} 
has no ground state solution for small $\varepsilon>0$. Suppose for the contrary that \eqref{e12} 
has a ground state $u_\varepsilon$ for all  $\varepsilon>0$. Then 
$$
v_\varepsilon(x)=\varepsilon^{-\frac{N(2+\alpha)}{4\alpha}}u_\varepsilon(\varepsilon^{-\frac{1}{2}}x)
$$ 
is a ground state of the equation \eqref{e32}.
  By the Nehari and Poho\v zaev identities,  we obtain
$$\begin{aligned}
\int_{\mathbb R^N}|\nabla v_{\varepsilon}|^2+\int_{\mathbb R^N}|v_{\varepsilon}|^2
=&\int_{\mathbb R^N}(I_{\alpha}*(|v_{\varepsilon}|^{\frac{N+\alpha}{N}}
+\varepsilon^{-\frac{(N+\alpha)(2+\alpha)}{4\alpha}}G(\varepsilon^{\frac{N(2+\alpha)}{4\alpha}}v_{\varepsilon})))\\
&\quad \cdot(\frac{N+\alpha}{N}|v_{\varepsilon}|^{\frac{N+\alpha}{N}}
+\varepsilon^{-\frac{(N+\alpha)(2+\alpha)}{4\alpha}}
g(\varepsilon^{\frac{N(2+\alpha)}{4\alpha}}v_{\varepsilon})\varepsilon^{\frac{N(2+\alpha)}{4\alpha}}v_{\varepsilon})dx\\
\end{aligned}
$$
and 
$$
\begin{aligned}
\frac{1}{2^*}\int_{\mathbb R^N}|\nabla v_{\varepsilon}|^2+\frac{1}{2}\int_{\mathbb R^N}|v_{\varepsilon}|^2
=&\frac{N+\alpha}{2N}\int_{\mathbb R^N}(I_{\alpha}*(|v_{\varepsilon}|^{\frac{N+\alpha}{N}}
+\varepsilon^{-\frac{(N+\alpha)(2+\alpha)}{4\alpha}}G(\varepsilon^{\frac{N(2+\alpha)}{4\alpha}}v_{\varepsilon})))\\
&\qquad \qquad \cdot(|v_{\varepsilon}|^{\frac{N+\alpha}{N}}
+\varepsilon^{-\frac{(N+\alpha)(2+\alpha)}{4\alpha}}G(\varepsilon^{\frac{N(2+\alpha)}{4\alpha}}v_{\varepsilon}))dx.\\
\end{aligned}
$$
Therefore,  we obtain
$$
\tilde m_\varepsilon=\frac{2+\alpha}{2(N+\alpha)}\int_{\mathbb R^N}|\nabla v_\varepsilon|^2+\frac{\alpha}{2(N+\alpha)}\int_{\mathbb R^N}|v_\varepsilon|^{2}.
$$
Hence, $\|\nabla v_\varepsilon\|_2$ and $\|v_\varepsilon\|_2$ are bounded.

On the other hand, by the Nehari and the Poho\v zaev identities,  we have
\begin{equation}\label{e37}
\begin{aligned}
\varepsilon^{\frac{(N+\alpha)(2+\alpha)}{4\alpha}}\|\nabla v_{\varepsilon}\|_2^2
&=\int_{\mathbb R^N}(I_{\alpha}*(|v_{\varepsilon}|^{\frac{N+\alpha}{N}}
+\varepsilon^{-\frac{(N+\alpha)(2+\alpha)}{4\alpha}}G(\varepsilon^{\frac{N(2+\alpha)}{4\alpha}}v_{\varepsilon})))\\
&\quad \cdot\left(\frac{N}{2}
g(\varepsilon^{\frac{N(2+\alpha)}{4\alpha}}v_{\varepsilon})\varepsilon^{\frac{N(2+\alpha)}{4\alpha}}v_{\varepsilon}
-\frac{N+\alpha}{2}G(\varepsilon^{\frac{N(2+\alpha)}{4\alpha}}v_{\varepsilon})\right)dx.\\
\end{aligned}
\end{equation}
Since $q\in [\frac{N+\alpha+4}{N}, \frac{N+\alpha}{N-2})$, by {\bf (H1)} and {\bf (H2)}, for any $\eta>\limsup_{s\to 0}g(s)/s^{\frac{\alpha+4}{N}}$ and any  small $\delta>0$, there exists $C>0$ such that
\begin{equation}\label{e38}
g(s)\le \eta s^{\frac{\alpha+4}{N}}+Cs^{q-1}+\delta s^{\frac{2+\alpha}{N-2}},\quad 
G(s)\le \eta s^{\frac{N+\alpha+4}{N}}+Cs^{q}+\delta s^{\frac{N+\alpha}{N-2}},
\end{equation}
and
\begin{equation}\label{e39}
\left|\frac{N}{2}g(s)s-\frac{N+\alpha}{2}G(s)\right|\le \eta s^{\frac{N+\alpha+4}{N}}+Cs^{q}+\delta s^{\frac{N+\alpha}{N-2}}, \quad \forall \ s\ge 0.
\end{equation}
By the Hardy-Littlewood-Sobolev, the Gagliardo-Nirenberg inequality, and the boundedness of $\{v_\varepsilon\}$ in $H^1(\mathbb R^N)$, we have
$$
\int_{\mathbb R^N}(I_{\alpha}*|v_{\varepsilon}|^{\frac{N+\alpha}{N}})|v_\varepsilon|^{\frac{N+\alpha+4}{N}} \lesssim\|\nabla v_\varepsilon\|_2^2,
$$
$$
\int_{\mathbb R^N}(I_{\alpha}*|v_{\varepsilon}|^{\frac{N+\alpha}{N}})|v_\varepsilon|^{q} \lesssim \|\nabla v_\varepsilon\|_2^{\frac{Nq-N-\alpha}{2}},
$$
$$
\int_{\mathbb R^N}(I_{\alpha}*|v_{\varepsilon}|^{\frac{N+\alpha}{N}})|v_\varepsilon|^{\frac{N+\alpha}{N-2}} \lesssim \|\nabla v_\varepsilon\|_2^{\frac{N+\alpha}{N-2}}.
$$
Therefore, we obtain 
\begin{equation}\label{e310}
\begin{array}{rl}
&\int_{\mathbb R^N}(I_{\alpha}\ast |v_{\varepsilon}|^{\frac{N+\alpha}{N}})\left(\frac{N}{2}
g(\varepsilon^{\frac{N(2+\alpha)}{4\alpha}}v_{\varepsilon})\varepsilon^{\frac{N(2+\alpha)}{4\alpha}}v_{\varepsilon}
-\frac{N+\alpha}{2}G(\varepsilon^{\frac{N(2+\alpha)}{4\alpha}}v_{\varepsilon})\right)dx\\
&\lesssim \eta \varepsilon^{\frac{(N+\alpha+4)(2+\alpha)}{4\alpha}}\|\nabla v_\varepsilon\|_2^2+ \varepsilon^{\frac{Nq(2+\alpha)}{4\alpha}}\|\nabla v_\varepsilon\|_2^{\frac{Nq-N-\alpha}{2}}
+\delta \varepsilon^{\frac{N(2+\alpha)(N+\alpha)}{4\alpha(N-2)}}\|\nabla v_\varepsilon\|_2^{\frac{N+\alpha}{N-2}}.
\end{array}
\end{equation}
Let $r,s\in [\frac{N+\alpha+4}{N}, \frac{N+\alpha}{N-2})$, then by the Hardy-Littlewood-Sobolev, the Gagliardo-Nirenberg inequality, and the boundedness of $\{v_\varepsilon\}$ in $H^1(\mathbb R^N)$, we get
$$
\int_{\mathbb R^N}(I_\alpha\ast |v_\varepsilon|^r)|v_\varepsilon|^s\le C\|\nabla v_\varepsilon\|_2^{\frac{N(r+s)}{2}-N-\alpha}.
$$
Therefore, we have 
\begin{equation}\label{e311}
\begin{array}{rl}
&\int_{\mathbb R^N}(I_{\alpha}\ast G(\varepsilon^{\frac{N(2+\alpha)}{4\alpha}}v_{\varepsilon}))
\left(\frac{N}{2}
g(\varepsilon^{\frac{N(2+\alpha)}{4\alpha}}v_{\varepsilon})\varepsilon^{\frac{N(2+\alpha)}{4\alpha}}v_{\varepsilon}
-\frac{N+\alpha}{2}G(\varepsilon^{\frac{N(2+\alpha)}{4\alpha}}v_{\varepsilon})\right)dx\\
&\lesssim \eta^2 \varepsilon^{\frac{(N+\alpha+4)(2+\alpha)}{2\alpha}}\|\nabla v_\varepsilon\|_2^4+\eta \varepsilon^{\frac{(N+\alpha+4+Nq)(2+\alpha)}{4\alpha}}\|\nabla v_\varepsilon\|_2^{\frac{Nq+4-N-\alpha}{2}}\\
&+\eta \delta \varepsilon^{\frac{(N+\alpha+4+N2_\alpha^*)(2+\alpha)}{4\alpha}}\|\nabla v_\varepsilon\|_2^{\frac{N2_\alpha^*+4-N-\alpha}{2}}
+\varepsilon^{\frac{Nq(2+\alpha)}{2\alpha}}\|\nabla v_\varepsilon\|_2^{Nq-N-\alpha}\\
&+\delta \varepsilon^{\frac{N(2_\alpha^*+q)(2+\alpha)}{4\alpha}} \|\nabla v_\varepsilon\|_2^{\frac{N(2_\alpha^*+q)-2N-2\alpha}{2}}+\delta^2\varepsilon^{\frac{N2_\alpha^*(2+\alpha)}{2\alpha}}\|\nabla v_\varepsilon\|_2^{N2_\alpha^*-N-\alpha},
\end{array}
\end{equation}
where  and in what follows $2^*_\alpha=\frac{N+\alpha}{N-2}$. Thus, we get
$$
\begin{array}{rl}
&\varepsilon^{\frac{(N+\alpha)(2+\alpha)}{4\alpha}}\|\nabla v_{\varepsilon}\|_2^2\\
&\lesssim \eta \varepsilon^{\frac{(N+\alpha+4)(2+\alpha)}{4\alpha}}\|\nabla v_\varepsilon\|_2^2+ \varepsilon^{\frac{Nq(2+\alpha)}{4\alpha}}\|\nabla v_\varepsilon\|_2^{\frac{Nq-N-\alpha}{2}}\\
&\mbox{} +\delta \varepsilon^{\frac{N(2+\alpha)(N+\alpha)}{4\alpha(N-2)}}\|\nabla v_\varepsilon\|_2^{\frac{N+\alpha}{N-2}}
+\eta^2 \varepsilon^{\frac{(N+\alpha+8)(2+\alpha)}{4\alpha}}\|\nabla v_\varepsilon\|_2^4\\
&\mbox{} +\eta \varepsilon^{\frac{(4+Nq)(2+\alpha)}{4\alpha}}\|\nabla v_\varepsilon\|_2^{\frac{Nq+4-N-\alpha}{2}}
+\eta \delta \varepsilon^{\frac{(4+N2_\alpha^*)(2+\alpha)}{4\alpha}}\|\nabla v_\varepsilon\|_2^{\frac{N2_\alpha^*+4-N-\alpha}{2}}\\
&\mbox{} +\varepsilon^{\frac{(2Nq-N-\alpha)(2+\alpha))}{4\alpha}}\|\nabla v_\varepsilon\|_2^{Nq-N-\alpha}+\delta \varepsilon^{\frac{(N(2_\alpha^*+q)-N-\alpha)(2+\alpha)}{4\alpha}} \|\nabla v_\varepsilon\|_2^{\frac{N(2_\alpha^*+q)-2N-2\alpha}{2}}\\
&\mbox{} +\delta^2\varepsilon^{\frac{(2N2_\alpha^*-N-\alpha)(2+\alpha)}{4\alpha}}\|\nabla v_\varepsilon\|_2^{N2_\alpha^*-N-\alpha},
\end{array}
$$
which implies that
\begin{equation}\label{e312}
\begin{array}{rcl}
1&\lesssim &\eta \varepsilon^{\frac{2+\alpha}{\alpha}}+ \varepsilon^{\frac{(Nq-N-\alpha)(2+\alpha)}{4\alpha}}\|\nabla v_\varepsilon\|_2^{\frac{Nq-N-\alpha-4}{2}}+\delta \varepsilon^{\frac{(N2_\alpha^*-N-\alpha)(2+\alpha)}{4\alpha}}\|\nabla v_\varepsilon\|_2^{2_\alpha^*-2}\\
&\mbox{}&+\eta^2 \varepsilon^{\frac{2(2+\alpha)}{\alpha}}\|\nabla v_\varepsilon\|_2^2+\eta \varepsilon^{\frac{(4+Nq-N-\alpha)(2+\alpha)}{4\alpha}}\|\nabla v_\varepsilon\|_2^{\frac{Nq-N-\alpha}{2}}\\
&\mbox{}&
+\eta \delta \varepsilon^{\frac{(4+N2_\alpha^*-N-\alpha)(2+\alpha)}{4\alpha}}\|\nabla v_\varepsilon\|_2^{\frac{N2_\alpha^*-N-\alpha}{2}}+\varepsilon^{\frac{(Nq-N-\alpha)(2+\alpha)}{2\alpha}}\|\nabla v_\varepsilon\|_2^{Nq-N-\alpha-2}\\
&\mbox{}&+\delta \varepsilon^{\frac{(N(2_\alpha^*+q)-2N-2\alpha)(2+\alpha)}{4\alpha}} \|\nabla v_\varepsilon\|_2^{\frac{N(2_\alpha^*+q)-2N-2\alpha-4}{2}}\\
&\mbox{}&+\delta^2\varepsilon^{\frac{(N2_\alpha^*-N-\alpha)(2+\alpha)}{2\alpha}}\|\nabla v_\varepsilon\|_2^{N2_\alpha^*-N-\alpha-2}.
\end{array}
\end{equation}
By the boundedness of $\{v_\varepsilon\}$ in $H^1(\mathbb R^N)$, we conclude $\varepsilon>\tilde \varepsilon_q$ for some $\tilde \varepsilon_q>0$. This leads to a contradiction and establish the claim.
Therefore, we conclude that  $\tilde m_\varepsilon\ge \frac{\alpha}{2N}(\frac{N}{N+\alpha}S_\alpha)^{\frac{N+\alpha}{\alpha}}$ for small $\varepsilon>0$. Thus, we get
$\tilde m_\varepsilon=\frac{\alpha}{2N}(\frac{N}{N+\alpha}S_\alpha)^{\frac{N+\alpha}{\alpha}}$  for small  $\varepsilon>0$. Let
\begin{equation}\label{e313}
\varepsilon_q:=\sup\left\{\varepsilon>0 \left | \  \tilde m_\varepsilon=\frac{\alpha}{2N}(\frac{N}{N+\alpha}S_\alpha)^{\frac{N+\alpha}{\alpha}} \right.\right\}.
\end{equation}
Then $\varepsilon_q>0$ for $\frac{N+\alpha+4}{N}\le q<\frac{N+\alpha}{N-2}$. It is well known (cf.\cite{LLT2022}) that
$\tilde m_\varepsilon<\frac{\alpha}{2N}(\frac{N}{N+\alpha}S_\alpha)^{\frac{N+\alpha}{\alpha}}$ for sufficiently large $\varepsilon>0$. Therefore, we have $0<\varepsilon_q<+\infty$. Hence, $\tilde m_\varepsilon=\frac{\alpha}{2N}(\frac{N}{N+\alpha}S_\alpha)^{\frac{N+\alpha}{\alpha}}$  for $\varepsilon\in (0,\varepsilon_q)$ and $\tilde m_\varepsilon<\frac{\alpha}{2N}(\frac{N}{N+\alpha}S_\alpha)^{\frac{N+\alpha}{\alpha}}$  for $\varepsilon>\varepsilon_q.$  If there exists $\varepsilon\in (0,\varepsilon_q)$ such that $\tilde m_\varepsilon$ is attained, then arguing as in \cite[Lemma 3.3]{WW2023}, we obtain  that $\tilde m_{\varepsilon'}<\tilde m_\varepsilon=\frac{\alpha}{2N}(\frac{N}{N+\alpha}S_\alpha)^{\frac{N+\alpha}{\alpha}}$  for $\varepsilon'>\varepsilon,$ which contradicts to the definition of $\varepsilon_q$.  Thus, we conclude that \eqref{e12} has no ground state solution for $\varepsilon\in (0,\varepsilon_q)$ and admits a ground state solution for $\varepsilon>\varepsilon_q$.

Finally, assume that $\lim_{s\to 0}g(s)/s^{\frac{\alpha+4}{N}}=0$, then we can choose $\eta>0$ sufficiently small in \eqref{e39} and \eqref{e312}.  Let $u_\varepsilon\in H^1(\mathbb R^N)$ be  a ground state of \eqref{e12} for $\varepsilon>\varepsilon_q$ and $v_\varepsilon(x)=\varepsilon^{-\frac{N(2+\alpha)}{4\alpha}}u_\varepsilon(\varepsilon^{-\frac{1}{2}}x)$, then $v_\varepsilon$ is bounded, hence $v_\varepsilon\rightharpoonup v_{\varepsilon_q}$ weakly in $H^1(\mathbb R^N)$ and $v_\varepsilon\to v_{\varepsilon_q}$  in $L^s(\mathbb R^N)$ as $\varepsilon\to \varepsilon_q^+$ for any $s\in (2,2^*)$.

If $q\in (\frac{N+\alpha+4}{N},\frac{N+\alpha}{N-2})$, then choose $\eta>0$ small enough, by \eqref{e312} and the boundedness of $\{v_\varepsilon\}$, it follows that there exists two constants $C_1,C_2>0$ such that
$$
C_1\le \|\nabla v_\varepsilon\|_2\le C_2, \quad {\rm for} \  \varepsilon>\varepsilon_q \ {\rm with} \ \varepsilon-\varepsilon_q>0 \ {\rm small}.
$$
Choosing $\delta>0$ sufficiently small,  by the boundedness of $\{v_\varepsilon\}$, then it follows \eqref{e37}, \eqref{e310} and \eqref{e311} that
\begin{equation}\label{e314}
\int_{\mathbb R^N}(I_\alpha\ast (|v_\varepsilon|^{\frac{N+\alpha}{N}}+|v_\varepsilon|^q))|v_\varepsilon|^q\sim 1, \quad {\rm as} \ \varepsilon\to \varepsilon_q^+.
\end{equation}
Since $\frac{N+\alpha}{N}<q<\frac{N+\alpha}{N-2}$, we choose a positive number $\delta_q$ such that 
$$
0<\delta_q<\min\left\{\frac{4N}{N+2}, \frac{2N}{q}\left(q-\frac{N+\alpha}{N}\right)\right\}.
$$
Then it is easy to see that 
\begin{equation}\label{e315}
2<\frac{2N-\delta_q}{N-\delta_q}<\frac{2N}{N-2}, \quad 2\le \frac{(2N-\delta_q)q}{N+\alpha}\le \frac{2N}{N-2}.
\end{equation}
By the Hardy-Littlewood-Sobolev inequality, we get 
$$
\begin{array}{rl}
&\int_{\mathbb R^N}(I_\alpha\ast |v_\varepsilon|^{\frac{N+\alpha}{N}})|v_\varepsilon|^q-
\int_{\mathbb R^N}(I_\alpha\ast |v_{\varepsilon_q}|^{\frac{N+\alpha}{N}})|v_{\varepsilon_q}|^q\\
&=\int_{\mathbb R^N}(I_\alpha\ast |v_\varepsilon|^{\frac{N+\alpha}{N}})(|v_\varepsilon|^q-|v_{\varepsilon_q}|^q)+
\int_{\mathbb R^N}(I_\alpha\ast (|v_\varepsilon|^{\frac{N+\alpha}{N}}-|v_{\varepsilon_q}|^{\frac{N+\alpha}{N}}))|v_{\varepsilon_q}|^q\\
&\le C\|v_\varepsilon\|_2^{\frac{N+\alpha}{N}}\|v_\varepsilon-v_{\varepsilon_q}\|_{\frac{2Nq}{N+\alpha}}^q+C\|v_\varepsilon-v_{\varepsilon_q}\|_{\frac{2N-\delta_q}{N-\delta_q}}^{\frac{N+\alpha}{N}}\|v_{\varepsilon_q}\|_{\frac{(2N-\delta_q)q}{N+\alpha}}^q.
\end{array}
$$
 Therefore, we get 
 $$
 \lim_{\varepsilon\to \varepsilon_q^+}\int_{\mathbb R^N}(I_\alpha\ast |v_\varepsilon|^{\frac{N+\alpha}{N}})|v_\varepsilon|^q=
\int_{\mathbb R^N}(I_\alpha\ast |v_{\varepsilon_q}|^{\frac{N+\alpha}{N}})|v_{\varepsilon_q}|^q.
$$
Similarly, we have 
$$
 \lim_{\varepsilon\to \varepsilon_q^+}\int_{\mathbb R^N}(I_\alpha\ast |v_\varepsilon|^q)|v_\varepsilon|^q=
\int_{\mathbb R^N}(I_\alpha\ast |v_{\varepsilon_q}|^q)|v_{\varepsilon_q}|^q.
$$
Thus, by \eqref{e314}, we have
\begin{equation}\label{e316}
\lim_{\varepsilon\to\varepsilon_q^+}\int_{\mathbb R^N}(I_\alpha\ast (|v_{\varepsilon}|^{\frac{N+\alpha}{N}}+|v_{\varepsilon}|^q))|v_{\varepsilon}|^q=\int_{\mathbb R^N}(I_\alpha\ast (|v_{\varepsilon_q}|^{\frac{N+\alpha}{N}}+|v_{\varepsilon_q}|^q))|v_{\varepsilon_q}|^q\not=0.
\end{equation}
Hence, we conclude 
that $v_{\varepsilon_q}\not=0$.  
It is standard to show that $v_\varepsilon\to v_{\varepsilon_q}$ in $H^1(\mathbb R^N)$  as $\varepsilon\to \varepsilon_q^+$ up to a subsequence, then it follows that $v_{\varepsilon_q}\in \tilde{\mathcal{N}}_{\varepsilon_q}$.
Thus $\tilde m_{\varepsilon_q}$ is attained by $v_{\varepsilon_q}$, and $v_{\varepsilon_q}$ is a  positive ground state solution of  the equation \eqref{e32} with $\varepsilon=\varepsilon_q$.
The proof is complete.
\end{proof}

\begin{corollary}  Assume that
$$
g(u)=\sum_{i=1}^kc_i|u|^{q_i-2}u, \quad c_i>0, \quad q_i\in \Big(\frac{N+\alpha}{N},\frac{N+\alpha}{N-2}\Big).
$$
Put $q:=\min\{q_i \ | \ i=1,2, \cdots, k\}$. Then the following statements hold true:

{\rm (1)} if $q\ge \frac{N+\alpha+4}{N}$, then there exists $\varepsilon_q>0$ such that  for any $\varepsilon\in (0,\varepsilon_q)$, the problem  \eqref{e12} 
has no  ground state, while for any $\varepsilon\ge \varepsilon_q$  $(\varepsilon>\varepsilon_q$ in the case that $q=\frac{N+\alpha+4}{N})$,  the problem  \eqref{e12} 
admits a positive  ground state $u_\varepsilon\in H^1(\mathbb R^N)$;

{\rm (2)} if $q<\frac{N+\alpha+4}{N}$, then for any $\varepsilon>0$, the problem  \eqref{e12} 
admits a positive  ground state $u_\varepsilon\in H^1(\mathbb R^N)$.
\end{corollary}




\noindent
\textbf{Remark 3.1.}  Assume  {\bf (H1)} and {\bf (H2)} hold, $q\ge \frac{N+\alpha+4}{N}$. If ${\bf (H3)}$ is not true,  then there exist $0<\varepsilon_q^1\le \varepsilon_q^2<\infty$ such that  for any $\varepsilon\in (0,\varepsilon_q^1)$, the problem \eqref{e12} has no  ground state, while if $\varepsilon> \varepsilon_q^2$,  the problem \eqref{e12} admits a positive  ground state $u_\varepsilon\in H^1(\mathbb R^N)$, which is  radially symmetric and radially nonincreasing. Moreover,
if  $\lim_{s\to 0}g(s)/s^{\frac{\alpha+4}{N}}=0$ and $q\in (\frac{N+\alpha+4}{N}, \frac{N+\alpha}{N-2})$, then, as $\varepsilon\to \varepsilon_q^{2+}$, up to a subsequence, the rescaled family of ground states $v_\varepsilon(x)=\varepsilon^{-\frac{N(2+\alpha)}{4\alpha}}u_\varepsilon(\varepsilon^{-\frac{1}{2}}x)$ converges in $H^1(\mathbb R^N)$ to a positive solution of the equation \eqref{e32} with $\varepsilon=\varepsilon_q^2$.
\smallskip

\section{Proof of Theorem \ref{t11}}
In this section, we assume that $q<\frac{N+\alpha+4}{N}$ and study the asymptotic behavior of positive ground state solutions $u_\varepsilon$ of \eqref{e12} as $\varepsilon \to 0$.
Let 
$$\sigma=\frac{(2+\alpha)[Nq-(N+\alpha)]}{\alpha(N+\alpha+4-Nq)},$$ 
we   consider the following  rescalings 
\begin{equation}\label{e41}
w(x)=\varepsilon^{-\frac{N\sigma}{4}}v(\varepsilon^{-\frac{\sigma}{2}}x), \qquad v(x)=\varepsilon^{-\frac{N(2+\alpha)}{4\alpha}}u(\varepsilon^{-\frac{1}{2}} x),
\end{equation}
then equation \eqref{e32} transforms into the equation
\begin{equation}\label{e42}
-\varepsilon^{\sigma}\Delta w+w=\left(I_{\alpha}*(|w|^{\frac{N+\alpha}{N}} +\varepsilon_1^{-1}G(\varepsilon_2w))\right)
\left(\frac{N+\alpha}{N}|w|^{\frac{N+\alpha}{N}-2}w
+\varepsilon_1^{-1}\varepsilon_2g(\varepsilon_2w)\right),
\end{equation}
where $$\varepsilon_1=\varepsilon^{\frac{(N+\alpha)(2+\alpha)}{4\alpha}+\frac{(N+\alpha)\sigma}{4}},\ \ \ \varepsilon_2=\varepsilon^{\frac{N(2+\alpha)}{4\alpha}+\frac{N\sigma}{4}}.$$

It is easy to check  that 
\begin{equation}\label{e43}
\varepsilon^{\sigma}=\varepsilon_1^{-1}\varepsilon_2^{q}.
\end{equation}
The corresponding energy functional of \eqref{e42} is given by
$$\begin{aligned}
J_{\varepsilon}(w)=\frac{1}{2}\int_{\mathbb R^N}\varepsilon^{\sigma}|\nabla w|^2+|w|^2-\frac{1}{2}\int_{\mathbb R^N}
&(I_{\alpha}*(|w|^{\frac{N+\alpha}{N}}+\varepsilon_1^{-1}G(\varepsilon_2w)))\\
\cdot&(|w|^{\frac{N+\alpha}{N}}+\varepsilon_1^{-1}G(\varepsilon_2w))dx.\\
\end{aligned}$$

Also, we denote the Nehari manifold and Poho\v{z}aev manifold  corresponding to \eqref{e47} by $\mathcal{N}_{\varepsilon}$ and $\mathcal{P}_{\varepsilon}$, respectively. That is,
$$\mathcal{N}_{\varepsilon}:=\{w \in H^1(\mathbb R^N)\backslash\{0\}\mid N_{\varepsilon}(w)=0 \},$$ where
$$\begin{aligned}
N_{\varepsilon}(w)=\int_{\mathbb R^N}\varepsilon^{\sigma}|\nabla w|^2+|w|^2
-\int_{\mathbb R^N}&(I_{\alpha}*(|w|^{\frac{N+\alpha}{N}}+\varepsilon_1^{-1}G(\varepsilon_2w)))\\
\cdot&(\frac{N+\alpha}{N}|w|^{\frac{N+\alpha}{N}}+\varepsilon_1^{-1}\varepsilon_2g(\varepsilon_2w)w)dx,\\
\end{aligned}$$
$$\mathcal{P}_{\varepsilon}:=\{w \in H^1(\mathbb R^N)\backslash\{0\}\mid P_{\varepsilon}(w)=0\},$$
where
$$\begin{aligned}
P_{\varepsilon}(w)&=\frac{(N-2)\varepsilon^{\sigma}}{2}\int_{\mathbb R^N}|\nabla w|^2+\frac{N}{2}\int_{\mathbb R^N}|w|^2\\
&\ \ \ -\frac{N+\alpha}{2}\int_{\mathbb R^N}(I_{\alpha}*(|w|^{\frac{N+\alpha}{N}}
+\varepsilon_1^{-1}G(\varepsilon_2w)))(|w|^{\frac{N+\alpha}{N}}+\varepsilon_1^{-1}G(\varepsilon_2w))dx.\\
\end{aligned}$$

As $\varepsilon \to 0$, the limit of the equation \eqref{e42} is the Hardy-Littlewood-Sobolev  critical equation \eqref{e19}.
The corresponding  energy functional is given by
$$J_0(w)=\frac{1}{2}\int_{\mathbb R^N}|w|^2-\frac{1}{2}\int_{\mathbb R^N}(I_{\alpha}*|w|^{\frac{N+\alpha}{N}})|w|^{\frac{N+\alpha}{N}},$$
and the corresponding Nehari and Poho\v zaev manifolds are defined by
$$\mathcal{N}_{0}=\mathcal{P}_{0}:=\left\{w \in H^1(\mathbb R^N)\backslash\{0\} \left | \
\int_{\mathbb R^N}|w|^2=\frac{N+\alpha}{N}\int_{\mathbb R^N}(I_{\alpha}*|w|^{\frac{N+\alpha}{N}})|w|^{\frac{N+\alpha}{N}}\right\}.\right. $$

In what follows, we set 
\begin{equation}\label{e44}
\mathbb D(u):=\int_{\mathbb R^N}(I_{\alpha}*|u|^{\frac{N+\alpha}{N}})|u|^{\frac{N+\alpha}{N}}.
\end{equation}
It is standard to verify the following result.
\begin{lemma}\label{lem41}
Let $\varepsilon>0$, $u \in H^1(\mathbb R^{N})$ and $v$ is the rescaling \eqref{e41} of $u$, $v \in H^1(\mathbb R^{N})$ and $w$ is the rescaling \eqref{e46} of $v$. Then
\begin{itemize}
\item[(i)] \ \ $\varepsilon^{-\frac{N}{\alpha}}\|u\|_2^2=\|v\|_2^2=\|w\|_2^2,\ \ \
            \varepsilon^{-\frac{N+\alpha}{\alpha}}\|\nabla u\|_2^2=\|\nabla v\|_2^2=\varepsilon^{\sigma}\|\nabla w\|_2^2,$
\item[(ii)] \  \ $ \varepsilon^{-\frac{N+\alpha}{\alpha}}\mathbb D(u)=\mathbb D(v)=\mathbb D(w).$
\end{itemize}
\end{lemma}
Moreover, arguing as in \cite{LMZ2019, LM2019}, it is easy to show that
$$
\tilde m_0=\inf_{w \in \mathcal{N}_{0}} J_0(w)=\inf_{w \in \mathcal{P}_{0}} J_0(w),
$$
and 
$$
\tilde m_{\varepsilon}=\inf_{w \in \mathcal{N}_{\varepsilon}} J_{\varepsilon}(w)=\inf_{w \in \mathcal{P}_{\varepsilon}} J_{\varepsilon}(w).
$$
Let $u_\varepsilon$ be a ground state of \eqref{e12} and set
$$
w_\varepsilon(x)=\varepsilon^{-\frac{N\sigma}{4}}v_\varepsilon(\varepsilon^{-\frac{\sigma}{2}}x), \qquad v_\varepsilon(x)=\varepsilon^{-\frac{N(2+\alpha)}{4\alpha}}u_\varepsilon(\varepsilon^{-\frac{1}{2}} x),
$$
then $w_\varepsilon$ is a ground state of \eqref{e47} and satisfies $w_\varepsilon\in \mathcal{N}_\varepsilon \cap \mathcal{P}_\varepsilon$ and $\tilde m_\varepsilon=J_\varepsilon(w_\varepsilon)$. 
\begin{lemma}\label{lem42}
Assume that  {\bf (H1)} and {\bf (H4)} hold, then the rescaled family of solutions $\{w_{\varepsilon}\}$ is bounded in $H^1(\mathbb R^N)$, and satisfies
\begin{equation}\label{e45}
\|\nabla w_{\varepsilon}\|_2^2=\frac{a[Nq-(N+\alpha)]}{2q}
\int_{\mathbb R^N}(I_{\alpha}*|w_{\varepsilon}|^{\frac{N+\alpha}{N}})|w_{\varepsilon}|^{q}+o_{\varepsilon}(1).
\end{equation}
\end{lemma}
\begin{proof}
First, since $G(s)\ge 0$ for all $s\ge 0$, by the min-max descriptions of $\tilde m_\varepsilon$ and $\tilde  m_0$, it is easy to show that $\tilde m_{\varepsilon}\leq \tilde m_{0}$. Also, according to $w_{\varepsilon}\in\mathcal{P}_{\varepsilon}$, we have
$$
\begin{aligned}
\tilde m_{0}\geq \tilde m_{\varepsilon}&=J_{\varepsilon}(w_{\varepsilon})=J_{\varepsilon}(w_{\varepsilon})-\frac{1}{N+\alpha}P_{\varepsilon}(w_{\varepsilon})\\
&=\frac{2+\alpha}{2(N+\alpha)}\int_{\mathbb R^N}\varepsilon^{\sigma}|\nabla w_{\varepsilon}|^2
+\frac{\alpha}{2(N+\alpha)}\int_{\mathbb R^N}|w_{\varepsilon}|^2.\\
\end{aligned}
$$
Hence,
we have $\{w_{\varepsilon}\}$ is bounded in $L^2(\mathbb R^N)$ and $\varepsilon^{\sigma}\|\nabla w_{\varepsilon}\|_2^2$ is bounded.
To show that $\{w_{\varepsilon}\}$ is bounded in $H^1(\mathbb R^N)$, it suffices to show that $\{w_{\varepsilon}\}$ is also bounded in $\mathcal{D}^{1,2}(\mathbb R^N)$.
Since $w_{\varepsilon} \in \mathcal{N}_{\varepsilon}\cap \mathcal{P}_{\varepsilon}$, we obtain
\begin{equation}\label{e46}
\begin{aligned}
\varepsilon^{\sigma}\|\nabla w_{\varepsilon}\|_2^2&=
\int_{\mathbb R^N}(I_{\alpha}*|w_{\varepsilon}|^{\frac{N+\alpha}{N}})
\left[\frac{N}{2}\varepsilon_1^{-1}g(\varepsilon_2w_{\varepsilon})\varepsilon_2w_{\varepsilon}-\frac{N+\alpha}{2}\varepsilon_1^{-1}G(\varepsilon_2w_{\varepsilon})\right]dx\\
&+\int_{\mathbb R^N}(I_{\alpha}*\varepsilon_1^{-1}G(\varepsilon_2w_{\varepsilon}))
\left[\frac{N}{2}\varepsilon_1^{-1}g(\varepsilon_2w_{\varepsilon})\varepsilon_2w_{\varepsilon}-\frac{N+\alpha}{2}\varepsilon_1^{-1}G(\varepsilon_2w_{\varepsilon})\right]dx,\\
\end{aligned}
\end{equation}
hence
\begin{equation}\label{e47}
\begin{aligned}
\|\nabla w_{\varepsilon}\|_2^2&=
\int_{\mathbb R^N}(I_{\alpha}*|w_{\varepsilon}|^{\frac{N+\alpha}{N}})
\left[\frac{N}{2}g(\varepsilon_2w_{\varepsilon})\varepsilon_2w_{\varepsilon}-\frac{N+\alpha}{2}G(\varepsilon_2w_{\varepsilon})\right]\varepsilon_2^{-q}dx\\
&+\int_{\mathbb R^N}(I_{\alpha}*\varepsilon_1^{-1}G(\varepsilon_2w_{\varepsilon}))
\left[\frac{N}{2}g(\varepsilon_2w_{\varepsilon})\varepsilon_2w_{\varepsilon}-\frac{N+\alpha}{2}G(\varepsilon_2w_{\varepsilon})\right]\varepsilon_2^{-q}dx.\\
\end{aligned}
\end{equation}
Put $g(s)=as^{q-1}+\tilde g(s)$, then $G(s)=\frac{a}{q}s^{q}+\tilde G(s)$, where $\tilde G(s)=\int_0^s\tilde g(\tau)d\tau$.
By {\bf (H1)} and {\bf (H4)}, we have
$$
\lim_{s\to 0}\tilde g(s)/|s|^{q-1}=\lim_{s\to 0}\tilde G(s)/|s|^{q}=0, 
$$
and 
$$
\lim_{s\to \infty}\tilde g(s)/|s|^{\frac{2+\alpha}{N-2}}=\lim_{s\to \infty}\tilde G(s)/|s|^{\frac{N+\alpha}{N-2}}=0.
$$
Therefore, for any $\delta>0$, there exists a constant $C_{\delta}>0$ such that
\begin{equation}\label{e48}
\left|\frac{N}{2}\tilde g(s)s-\frac{N+\alpha}{2}\tilde G(s)\right|\le \delta |s|^{q}+C_{\delta}|s|^{\frac{N+\alpha}{N-2}}.
\end{equation}
According to \eqref{e47}, we have
\begin{equation}\label{e49}
\begin{aligned}
\|\nabla w_{\varepsilon}\|_2^2
&=\frac{a[Nq-(N+\alpha)]}{2q}\int_{\mathbb R^N}(I_{\alpha}\ast |w_{\varepsilon}|^{\frac{N+\alpha}{N}})|w_{\varepsilon}|^{q}\\
&+\frac{a[Nq-(N+\alpha)]}{2q}\int_{\mathbb R^N}(I_{\alpha}\ast \varepsilon_1^{-1}G(\varepsilon_2 w_\varepsilon))|w_{\varepsilon}|^{q}\\
&+\int_{\mathbb R^N}(I_\alpha\ast (|w_\varepsilon|^{\frac{N+\alpha}{N}}+\varepsilon_1^{-1}G(\varepsilon_2 w_\varepsilon)))
\mbox{} \cdot \left[\frac{N}{2}
\tilde g(\varepsilon_2 w_\varepsilon)\varepsilon_2 w_\varepsilon-\frac{N+\alpha}{2}\tilde G(\varepsilon_2 w_\varepsilon)\right]\varepsilon_2^{-q}dx.
\end{aligned}
\end{equation}
Recall that $\varepsilon^{\sigma}\|\nabla w_{\varepsilon}\|_2^2$ is bounded, by the Hardy-Littlewood-Sobolev inequality and  the Gagliardo-Nirenberg inequality,  we get  
$$
\begin{array}{rl}
&\varepsilon_1^{-1}\varepsilon_2^{q}\int_{\mathbb R^N}(I_{\alpha}\ast |w_\varepsilon|^{q})|w_{\varepsilon}|^{q}dx
\lesssim \|\nabla w_\varepsilon\|_2^{Nq-(N+\alpha+2)}, 
\end{array}
$$
and
$$
\begin{array}{rl}
\varepsilon_1^{-1}\varepsilon_2^{\frac{N+\alpha}{N-2}}\int_{\mathbb R^N}(I_{\alpha}\ast |w_\varepsilon|^{\frac{N+\alpha}{N-2}})|w_{\varepsilon}|^{q}dx
&\lesssim \varepsilon_2^{\frac{N+\alpha}{N-2}-q}\|\nabla w_\varepsilon\|_2^{\frac{N}{2}(\frac{N+\alpha}{N-2}+q)-(N+\alpha+2)}\\
&\lesssim \varepsilon^{\frac{N(2+\alpha)}{4\alpha}(\frac{N+\alpha}{N-2}-q)}\|\nabla w_\varepsilon\|_2^{Nq-(N+\alpha+2)}.
\end{array}
$$
Therefore, it follows from \eqref{e48} that
\begin{equation}\label{e410}
\begin{array}{rl}
&\int_{\mathbb R^N}(I_{\alpha}\ast \varepsilon_1^{-1}G(\varepsilon_2 w_\varepsilon))|w_{\varepsilon}|^{q}\\
&\lesssim
\varepsilon_1^{-1}\varepsilon_2^{q}\int_{\mathbb R^N}(I_{\alpha}\ast |w_\varepsilon|^{q})|w_{\varepsilon}|^{q}
+\varepsilon_1^{-1}\varepsilon_2^{\frac{N+\alpha}{N-2}}\int_{\mathbb R^N}(I_{\alpha}\ast |w_\varepsilon|^{\frac{N+\alpha}{N-2}})|w_{\varepsilon}|^{q}\\
&\lesssim \|\nabla w_\varepsilon\|_2^{Nq-(N+\alpha+2)}
+\varepsilon^{\frac{N(2+\alpha)}{4\alpha}(\frac{N+\alpha}{N-2}-q)}\|\nabla w_\varepsilon\|_2^{Nq-(N+\alpha+2)},
\end{array}
\end{equation}
Similarly, we have 
\begin{equation}\label{e411}
\begin{array}{rl}
&\int_{\mathbb R^N}(I_\alpha\ast \varepsilon_1^{-1}G(\varepsilon_2 w_\varepsilon))\left |\frac{N}{2}\tilde g(\varepsilon_2 w_\varepsilon)\varepsilon_2 w_\varepsilon-\frac{N+\alpha}{2}\tilde G(\varepsilon_2 w_\varepsilon)\right |\varepsilon_2^{-q}dx\\
&\lesssim \varepsilon_1^{-1}\varepsilon_2^{q}\int_{\mathbb R^N}(I_{\alpha}\ast |w_\varepsilon|^{q})|w_{\varepsilon}|^{q}
+\varepsilon_1^{-1}\varepsilon_2^{\frac{N+\alpha}{N-2}}\int_{\mathbb R^N}(I_{\alpha}\ast |w_\varepsilon|^{\frac{N+\alpha}{N-2}})|w_{\varepsilon}|^{q}\\
&\qquad +\varepsilon_1^{-1}\varepsilon_2^{\frac{2(N+\alpha)}{N-2}-q}\int_{\mathbb R^N}
(I_{\alpha}\ast |w_\varepsilon|^{\frac{N+\alpha}{N-2}})|w_{\varepsilon}|^{\frac{N+\alpha}{N-2}}\\
&\lesssim \|\nabla w_\varepsilon\|_2^{Nq-(N+\alpha+2)}
+\varepsilon^{\frac{N(2+\alpha)}{4\alpha}(\frac{N+\alpha}{N-2}-q)}\|\nabla w_\varepsilon\|_2^{Nq-(N+\alpha+2)}\\
&\qquad +\varepsilon^{\frac{N(2+\alpha)[(N+\alpha)-(N-2)q]}{2\alpha(N-2)}}\|\nabla w_{\varepsilon}\|_2^{Nq-(N+\alpha+2)},
\end{array}
\end{equation}
and 
\begin{equation}\label{e412}
\begin{array}{rl}
&\int_{\mathbb R^N}(I_\alpha\ast |w_\varepsilon|^{\frac{N+\alpha}{N}})\left|\frac{N}{2}\tilde g(\varepsilon_2 w_\varepsilon)\varepsilon_2 w_\varepsilon-\frac{N+\alpha}{2}\tilde G(\varepsilon_2 w_\varepsilon)\right|\varepsilon_2^{-q}dx\\
&\lesssim \delta \int_{\mathbb R^N}(I_\alpha\ast |w_\varepsilon|^{\frac{N+\alpha}{N}})|w_\varepsilon|^{q}
+C_{\delta}\varepsilon_2^{\frac{N+\alpha}{N-2}-q}\int_{\mathbb R^N}(I_\alpha\ast |w_\varepsilon|^{\frac{N+\alpha}{N}})|w_\varepsilon|^{\frac{N+\alpha}{N-2}}\\
&\lesssim \delta \|\nabla w_{\varepsilon}\|_2^{\frac{Nq-(N+\alpha)}{2}}+\varepsilon^{\frac{(2+\alpha)(N+\alpha)}{2\alpha(N-2)}}\|\nabla w_{\varepsilon}\|_2^2
\end{array}
\end{equation}
Since 
$$
\int_{\mathbb R^N}(I_{\alpha}\ast |w_{\varepsilon}|^{\frac{N+\alpha}{N}})|w_{\varepsilon}|^{q}dx\lesssim \|\nabla w_{\varepsilon}\|_2^{\frac{Nq-(N+\alpha)}{2}},
$$
 and  note that $Nq-(N+\alpha+2)<2$ and $\frac{Nq-(N+\alpha)}{2}<2,$ it follows from \eqref{e49}--\eqref{e412} that 
 $\{w_\varepsilon\}$ is bounded in $D^{1,2}(\mathbb R^N)$. Hence, $\{w_\varepsilon\}$ is bounded in $H^1(\mathbb R^N)$. 

On the other hand,  by the Hardy-Littlewood-Sobolev inequality,  it follows from \eqref{e410}  that 
$$
\begin{array}{rl}
&\int_{\mathbb R^N}(I_{\alpha}\ast \varepsilon_1^{-1}G(\varepsilon_2 w_\varepsilon))|w_{\varepsilon}|^{q}\\
&\lesssim
\varepsilon_1^{-1}\varepsilon_2^{q}\int_{\mathbb R^N}(I_{\alpha}\ast |w_\varepsilon|^{q})|w_{\varepsilon}|^{q}dx
+\varepsilon_1^{-1}\varepsilon_2^{\frac{N+\alpha}{N-2}}\int_{\mathbb R^N}(I_{\alpha}\ast |w_\varepsilon|^{\frac{N+\alpha}{N-2}})|w_{\varepsilon}|^{q}\\
&\lesssim \varepsilon^\sigma \|w_\varepsilon\|_{\frac{2Nq}{N+\alpha}}^{2q}+\varepsilon^\sigma \varepsilon_2^{\frac{N+\alpha}{N-2}-q}\|w_\varepsilon\|_{2^*}^{\frac{N+\alpha}{N-2}}\|w_\varepsilon\|_{\frac{2Nq}{N+\alpha}}^q,
\end{array}
$$
which together with the  boundedness of $\{w_\varepsilon\}$ in $H^1(\mathbb R^N)$ implies that
$$
\lim_{\varepsilon\to 0}\int_{\mathbb R^N}(I_{\alpha}\ast \varepsilon_1^{-1}G(\varepsilon_2 w_\varepsilon))|w_{\varepsilon}|^{q}=0.
$$
Similarly, from \eqref{e411} and \eqref{e412} we get 
$$
\lim_{\varepsilon \to 0}
\int_{\mathbb R^N}(I_\alpha\ast \varepsilon_1^{-1}G(\varepsilon_2 w_\varepsilon))\left|\frac{N}{2}
\tilde g(\varepsilon_2 w_\varepsilon)\varepsilon_2 w_\varepsilon-\frac{N+\alpha}{2}\tilde G(\varepsilon_2 w_\varepsilon)\right|\varepsilon_2^{-q}dx=0,
$$
and 
$$
\limsup_{\varepsilon \to 0}
\int_{\mathbb R^N}(I_\alpha\ast |w_\varepsilon|^{\frac{N+\alpha}{N}})\left |\frac{N}{2}
\tilde g(\varepsilon_2 w_\varepsilon)\varepsilon_2 w_\varepsilon-\frac{N+\alpha}{2}\tilde G(\varepsilon_2 w_\varepsilon)\right|\varepsilon_2^{-q}dx\lesssim \delta,
$$
Because $\delta$ is arbitrary,  by \eqref{e49}, we obtain
$$\|\nabla w_{\varepsilon}\|_2^2=\frac{a[Nq-(N+\alpha)]}{2q}
\int_{\mathbb R^N}(I_{\alpha}*|w_{\varepsilon}|^{\frac{N+\alpha}{N}})|w_{\varepsilon}|^{q}+o_{\varepsilon}(1).$$
The proof is complete.
\end{proof}

For $w \in H^1(\mathbb R^N)\backslash\{0\}$, let
\begin{equation}\label{e413}
\tau(w)=\frac{N}{N+\alpha}\cdot\frac{\int_{\mathbb R^N}|w|^2}{\int_{\mathbb R^N}(I_{\alpha}*|w|^{\frac{N+\alpha}{N}})|w|^{\frac{N+\alpha}{N}}},
\end{equation}
then $\tau(w)^{\frac{N}{2\alpha}}w \in \mathcal{N}_{0}$ for any $w \in H^1(\mathbb R^N)\backslash\{0\}$ and $w \in \mathcal{N}_{0}$ if and only if $\tau(w)=1$.

\begin{lemma}\label{lem43}
Under the conditions {\bf(H1)} and {\bf(H4)}, then  
there exists a constant $C>0$ such that for any small $\varepsilon>0$, there holds
\begin{equation}\label{e414}
1-C\varepsilon^{\sigma} \leq \tau(w_{\varepsilon})\leq 1+C\varepsilon^{\sigma}.
\end{equation}
\end{lemma}
\begin{proof}
Since $w_{\varepsilon} \in \mathcal{N}_{\varepsilon}$,  it follows  that
\begin{equation}\label{e415}
\tau(w_{\varepsilon})=\frac{N}{N+\alpha}\cdot\frac{\int_{\mathbb R^N}|w_{\varepsilon}|^2}
{\int_{\mathbb R^N}(I_{\alpha}*|w_{\varepsilon}|^{\frac{N+\alpha}{N}})|w_{\varepsilon}|^{\frac{N+\alpha}{N}}}
=1+\tau(w_\varepsilon)H(w_\varepsilon),
\end{equation}
where
\begin{equation}\label{e416}
\begin{aligned}
&H(w_{\varepsilon})=\frac{\int_{\mathbb R^N}(I_{\alpha}*\varepsilon_1^{-1}G(\varepsilon_2w_{\varepsilon}))\varepsilon_1^{-1}\varepsilon_2g(\varepsilon_2w_{\varepsilon})w_{\varepsilon}}
{\int_{\mathbb R^N}|w_{\varepsilon}|^2}\\
+&\frac{\int_{\mathbb R^N}(I_{\alpha}*|w_{\varepsilon}|^{\frac{N+\alpha}{N}})(\frac{N+\alpha}{N}\varepsilon_1^{-1}G(\varepsilon_2w_{\varepsilon})
+\varepsilon_1^{-1}\varepsilon_2g(\varepsilon_2w_{\varepsilon})w_{\varepsilon})-\varepsilon^{\sigma}\int_{\mathbb R^N}|\nabla w_{\varepsilon}|^2}
{\int_{\mathbb R^N}|w_{\varepsilon}|^2}.
\end{aligned}
\end{equation}
It follows from {\bf (H1)} and {\bf (H4)} that $g(s)s, G(s)\le C_1s^q+C_2s^{\frac{N+\alpha}{N-2}}$ for all $s\ge 0$ and some $C_1,C_2>0$. Therefore, we get
\begin{equation}\label{e417}
\begin{aligned}
&\int_{\mathbb R^N}(I_{\alpha}*\varepsilon_1^{-1}G(\varepsilon_2w_{\varepsilon}))\varepsilon_1^{-1}\varepsilon_2g(\varepsilon_2w_{\varepsilon})w_{\varepsilon}\\
&\lesssim \varepsilon_1^{-2}\varepsilon_2^{2q}\int_{\mathbb R^N}(I_\alpha\ast |w_\varepsilon|^q)|w_\varepsilon|^q
+\varepsilon_1^{-2}\varepsilon_2^{q+\frac{N+\alpha}{N-2}}\int_{\mathbb R^N}(I_\alpha\ast |w_\varepsilon|^q)|w_\varepsilon|^{\frac{N+\alpha}{N-2}}\\
&+\varepsilon_1^{-2}\varepsilon_2^{2\frac{N+\alpha}{N-2}}\int_{\mathbb R^N}(I_\alpha\ast |w_\varepsilon|^{\frac{N+\alpha}{N-2}})|w_\varepsilon|^{\frac{N+\alpha}{N-2}},
\end{aligned}
\end{equation}
and 
\begin{equation}\label{e418}
\begin{aligned}
&\int_{\mathbb R^N}(I_{\alpha}\ast |w_\varepsilon|^{\frac{N+\alpha}{N}})(\frac{N+\alpha}{N}\varepsilon_1^{-1}G(\varepsilon_2w_{\varepsilon})
+\varepsilon_1^{-1}\varepsilon_2g(\varepsilon_2w_{\varepsilon})w_{\varepsilon})dx\\
&\lesssim \varepsilon_1^{-1}\varepsilon_2^{q}\int_{\mathbb R^N}(I_\alpha\ast |w_\varepsilon|^{\frac{N+\alpha}{N}})|w_\varepsilon|^q
+\varepsilon_1^{-1}\varepsilon_2^{\frac{N+\alpha}{N-2}}\int_{\mathbb R^N}(I_\alpha\ast |w_\varepsilon|^{\frac{N+\alpha}{N}})|w_\varepsilon|^{\frac{N+\alpha}{N-2}}.
\end{aligned}
\end{equation}
By the Hardy-Littlewood-Sobolev, the Sobolev and the interpolation inequalities, it is easy to see that  for any $t,s\in [\frac{N+\alpha}{N},\frac{N+\alpha}{N-2}]$, there holds
$$
\int_{\mathbb R^N}(I_\alpha \ast |w_\varepsilon|^t)|w_\varepsilon|^s\lesssim \|w_\varepsilon\|_2^{N+\alpha-\frac{t+s}{2}(N-2)}\|\nabla w_\varepsilon\|_2^{N\frac{t+s}{2}-N-\alpha}.
$$
Note that $N+\alpha-\frac{t+s}{2}(N-2)+N\frac{t+s}{2}-N-\alpha=t+s\in [\frac{2(N+\alpha)}{N},\frac{2(N+\alpha)}{N-2}]$, 
we have 
$$
\frac{N+\alpha-\frac{t+s}{2}(N-2)}{t+s}+\frac{N\frac{t+s}{2}-N-\alpha}{t+s}=1.
$$
 By Hardy-Littlewood-Sobolev,  Young inequalities and the boundedness of $\{w_\varepsilon\}$ in $H^1(\mathbb R^N)$, for $t+s\in (2\frac{N+\alpha}{N}, 2\frac{N+\alpha}{N-2})$, we have 
 \begin{equation}\label{e419}
\begin{aligned}
\int_{\mathbb R^N}(I_\alpha \ast |w_\varepsilon|^t)|w_\varepsilon|^s&\lesssim \frac{N+\alpha-\frac{t+s}{2}(N-2)}{t+s}\|w_\varepsilon\|_2^{t+s}+\frac{N\frac{t+s}{2}-N-\alpha}{t+s}\|\nabla w_\varepsilon\|_2^{t+s}\\
&\lesssim \|w_\varepsilon\|_2^2+\|\nabla w_\varepsilon\|_2^2.
\end{aligned}
\end{equation}
Therefore, by $\varepsilon^\sigma=\varepsilon_1^{-1}\varepsilon_2^q$, \eqref{e416}, \eqref{e417}, \eqref{e418}, \eqref{e419} and the boundedness of $\{w_\varepsilon\}$ in $H^1(\mathbb R^N)$, we have 
$$\begin{aligned}
\varepsilon^{-\sigma}H(w_\varepsilon)\|w_\varepsilon\|_2^2+\|\nabla w_\varepsilon\|_2^2\lesssim  &\|w_\varepsilon\|_2^{N+\alpha-\frac{1}{2}(\frac{N+\alpha}{N}+q)(N-2)}\|\nabla w_\varepsilon\|_2^{\frac{1}{2}N(\frac{N+\alpha}{N}+q)-N-\alpha}\\
&+\varepsilon_2^{\frac{N+\alpha}{N-2}-q}(\|w_\varepsilon\|_2^{\frac{N+\alpha}{N}+\frac{N+\alpha}{N-2}}+\|\nabla w_\varepsilon\|_2^{\frac{N+\alpha}{N}+\frac{N+\alpha}{N-2}})\\
&+\varepsilon^\sigma (\|w_\varepsilon\|_2^{2q}+\|\nabla w_\varepsilon\|_2^{2q})\\
&+\varepsilon^{\sigma}\varepsilon_2^{\frac{N+\alpha}{N-2}-q}(\|w_\varepsilon\|_2^{q+\frac{N+\alpha}{N-2}}+\|\nabla w_\varepsilon\|_2^{q+\frac{N+\alpha}{N-2}})\\
&+\varepsilon^\sigma\varepsilon_2^{2(\frac{N+\alpha}{N-2}-q)}\|\nabla w_\varepsilon\|^{\frac{2(N+\alpha)}{N-2}}\\
\lesssim  &\|w_\varepsilon\|_2^{N+\alpha-\frac{1}{2}(\frac{N+\alpha}{N}+q)(N-2)}\|\nabla w_\varepsilon\|_2^{\frac{1}{2}N(\frac{N+\alpha}{N}+q)-N-\alpha}\\
&+o_\varepsilon(1)(\|w_\varepsilon\|_2^2+\|\nabla w_\varepsilon\|_2^2).
\end{aligned}
$$
Note that $q<\frac{N+\alpha+4}{N}$ implies that $\frac{1}{2}N(\frac{N+\alpha}{N}+q)-N-\alpha<2$, it follows from  the boundedness of $\{w_\varepsilon\}$ in $H^1(\mathbb R^N)$ that  
$$\begin{aligned}
\varepsilon^{-\sigma}H(w_{\varepsilon}) 
\lesssim  &\frac{1}{\|w_\varepsilon\|_2^2}\|w_\varepsilon\|_2^{N+\alpha-\frac{1}{2}(\frac{N+\alpha}{N}+q)(N-2)}\|\nabla w_\varepsilon\|_2^{\frac{1}{2}N(\frac{N+\alpha}{N}+q)-N-\alpha}\\
&-\frac{\|\nabla w_\varepsilon\|_2^2}{\|w_\varepsilon\|_2^2}
+o_\varepsilon(1)(1+\frac{\|\nabla w_\varepsilon\|_2^2}{\|w_\varepsilon\|_2^2})\\ 
\lesssim & C\left(\frac{\|\nabla w_{\varepsilon}\|_2}{\|w_{\varepsilon}\|_2}\right)^{\frac{1}{2}N(\frac{N+\alpha}{N}+q)-N-\alpha}-\frac{1}{2}\left(\frac{\|\nabla w_{\varepsilon}\|_2}{\|w_{\varepsilon}\|_2}\right)^2+o_\varepsilon(1)<C<+\infty.
\end{aligned}$$
Therefore, we obtain $H(w_\varepsilon)\lesssim \varepsilon^\sigma$, and hence 
$$
\tau(w_{\varepsilon})=1+\tau(w_\varepsilon)H(w_\varepsilon)\le 1+C\varepsilon^{\sigma}\tau(w_{\varepsilon}),
$$
therefore, for small $\varepsilon>0$, there holds
$$\tau(w_{\varepsilon})\leq \frac{1}{1-C\varepsilon^{\sigma}}=\frac{1-C\varepsilon^{\sigma}+C\varepsilon^{\sigma}}{1-C\varepsilon^{\sigma}}\leq 1+C\varepsilon^{\sigma}.$$
On the  other hand,  from \eqref{e413} and  Hardy-Littlewood-Sobolev inequality, we get
$$\int_{\mathbb R^N}|w_{\varepsilon}|^2
=\frac{N+\alpha}{N}\tau(w_{\varepsilon})\int_{\mathbb R^N}(I_{\alpha}*|w_{\varepsilon}|^{\frac{N+\alpha}{N}})|w_{\varepsilon}|^{\frac{N+\alpha}{N}}
\leq C\tau(w_{\varepsilon})(\int_{\mathbb R^N}|w_{\varepsilon}|^2)^{\frac{N+\alpha}{N}},$$
and hence
$$\int_{\mathbb R^N}|w_{\varepsilon}|^2\geq(C\tau(w_{\varepsilon}))^{-\frac{N}{\alpha}}\geq C>0.$$
Therefore, by \eqref{e416} and  the boundedness of $\{w_{\varepsilon}\}$ in $H^1(\mathbb R^N)$, we have
$$
H(w_\varepsilon)\ge -\varepsilon^\sigma\frac{\int_{\mathbb R^N}|\nabla w_\varepsilon|^2}{\int_{\mathbb R^N}|w_\varepsilon|^2}\ge  -C\varepsilon^\sigma, 
$$
and hence 
$$
\tau(w_{\varepsilon})=1+\tau(w_\varepsilon)H(w_\varepsilon)\ge 1-C\varepsilon^{\sigma}\tau(w_{\varepsilon}),
$$
which together with the fact that $\tau(w_\varepsilon)\le 1+C\varepsilon^\sigma$ implies that  $\tau(w_{\varepsilon})\geq1-C\varepsilon^{\sigma}$.
The proof is complete.
\end{proof}

\begin{lemma}\label{lem44}
Under the conditions {\bf(H1)} and {\bf(H4)}, we have
$$\tilde m_0-\tilde m_{\varepsilon}\sim \varepsilon^{\sigma}\ \ \ \mbox{as}\ \ \ \varepsilon\ \to\ 0.$$
\end{lemma}
\begin{proof}
By {\bf (H4)}, Lemma \ref{lem25}, Lemma \ref{lem43} and the boundedness of $\{w_{\varepsilon}\}$ in $H^1(\mathbb R^N)$, we find that
$$\begin{aligned}
\tilde m_0&\leq \sup_{t\geq0}J_0((w_{\varepsilon})_t)=J_0((w_{\varepsilon})_{t_{\varepsilon}})\\
&\leq \sup_{t\geq0}J_{\varepsilon}((w_{\varepsilon})_t)-\frac{1}{2}\varepsilon^{\sigma}\int_{\mathbb R^N}|\nabla(w_{\varepsilon})_{t_{\varepsilon}}|^2\\
&\ \ \ +\frac{1}{2}\int_{\mathbb R^N}
(I_{\alpha}*(2|(w_{\varepsilon})_{t_{\varepsilon}}|^{\frac{N+\alpha}{N}}+\varepsilon_1^{-1}G(\varepsilon_2(w_{\varepsilon})_{t_{\varepsilon}})))
\varepsilon_1^{-1}G(\varepsilon_2(w_{\varepsilon})_{t_{\varepsilon}})dx\\
&\leq \tilde m_{\varepsilon}+\varepsilon^{\sigma}\Big(\tau(w_{\varepsilon})^{\frac{N+\alpha}{\alpha}}
C\int_{\mathbb R^N}(I_{\alpha}*|w_{\varepsilon}|^{\frac{N+\alpha}{N}})|w_{\varepsilon}|^{q}
-\frac{\tau(w_{\varepsilon})^{\frac{N-2}{\alpha}}}{2}\int_{\mathbb R^N}|\nabla w_{\varepsilon}|^2\Big)+o(\varepsilon^{\sigma})\\
&\leq  \tilde m_{\varepsilon}+C\varepsilon^{\sigma},
\end{aligned}$$
where
\begin{equation}\label{e420}
t_{\varepsilon}=\left(\frac{N}{N+\alpha}\cdot \frac{\int_{\mathbb R^N}|w_{\varepsilon}|^2}{\int_{\mathbb R^N}(I_{\alpha}*|w_{\varepsilon}|^{\frac{N+\alpha}{N}})|w_{\varepsilon}|^{\frac{N+\alpha}{N}}}\right)^{\frac{1}{\alpha}}=\tau(w_{\varepsilon})^{\frac{1}{\alpha}}.
\end{equation}

For each $\rho>0$, the family $W_{\rho}:=\rho^{-\frac{N}{2}}W_1(x/\rho)$ are radial ground states of \eqref{e19}.
Then
$$\begin{aligned}
g_0(\rho):=&\frac{a}{q}\int_{\mathbb R^N}(I_{\alpha}*|W_\rho|^{\frac{N+\alpha}{N}})|W_\rho|^{q}-\frac{1}{2}\int_{\mathbb R^N}|\nabla W_\rho|^2\\
=&\frac{a}{q}\rho^{\frac{N+\alpha}{2}-\frac{Nq}{2}}\int_{\mathbb R^N}(I_{\alpha}*|W_1|^{\frac{N+\alpha}{N}})|W_1|^{q}
-\frac{1}{2}\rho^{-2}\int_{\mathbb R^N}|\nabla W_1|^2.\\
\end{aligned}$$
Clearly, there exists $\rho_0=\rho(q)\in (0,\infty)$ with
$$\rho_0=\left(\frac{2q\int_{\mathbb R^N}|\nabla W_1|^2}
{a[Nq-(N+\alpha)]\int_{\mathbb R^N}(I_{\alpha}*|W_1|^{\frac{N+\alpha}{N}})|W_1|^{q}}\right)^{\frac{2}{N+\alpha+4-Nq}},$$
such that
$$
g_0(\rho_0)=\sup_{\rho>0}g_0(\rho)=\frac{N+\alpha+4-Nq}{2[Nq-(N+\alpha)]}\cdot\rho_0^{-2}\int_{\mathbb R^N}|\nabla W_1|^2>0.
$$
Let $W_0=W_{\rho_0}$,
then by {\bf (H4)} and Lemma \ref{lem25}, there exists $t_{\varepsilon} \in (0,\infty)$ with $t_{\varepsilon}W_0 \in \mathcal{N}_{\varepsilon}$ such that
\begin{equation}\label{e421}
\begin{aligned}
\tilde m_{\varepsilon}&\leq \sup_{t\geq0}J_{\varepsilon}(tW_0)=J_{\varepsilon}(t_{\varepsilon}W_0)\\
&\leq\sup_{t\geq 0}\left(\frac{t^2}{2}\int_{\mathbb R^N}|W_0|^2
-\frac{t^{\frac{2(N+\alpha)}{N}}}{2}\int_{\mathbb R^N}(I_{\alpha}*|W_0|^{\frac{N+\alpha}{N}})|W_0|^{\frac{N+\alpha}{N}}\right)\\
&\ -\varepsilon^{\sigma}\left(\frac{a}{q}t_{\varepsilon}^{\frac{N+\alpha}{N}+q}
\int_{\mathbb R^N}(I_{\alpha}*|W_0|^{\frac{N+\alpha}{N}})|W_0|^{q}-\frac{t_{\varepsilon}^2}{2}\int_{\mathbb R^N}|\nabla W_0|^2\right)+o(\varepsilon^{\sigma})\\
&=\tilde m_0-\varepsilon^{\sigma}\left(\frac{a}{q}t_{\varepsilon}^{\frac{N+\alpha}{N}+q}
\int_{\mathbb R^N}(I_{\alpha}*|W_0|^{\frac{N+\alpha}{N}})|W_0|^{q}-\frac{t_{\varepsilon}^2}{2}\int_{\mathbb R^N}|\nabla W_0|^2\right)+o(\varepsilon^{\sigma}).\\
\end{aligned}
\end{equation}
According to $t_{\varepsilon}W_0 \in \mathcal{N}_{\varepsilon}$, we have
$$\begin{aligned}
t_{\varepsilon}^2&\int_{\mathbb R^N}\varepsilon^{\sigma}|\nabla W_0|^2+|W_0|^2
=t_{\varepsilon}^{\frac{2(N+\alpha)}{N}}\cdot\frac{N+\alpha}{N}\int_{\mathbb R^N}(I_{\alpha}*|W_0|^{\frac{N+\alpha}{N}})|W_0|^{\frac{N+\alpha}{N}}\\
&+\int_{\mathbb R^N}(I_{\alpha}*|t_{\varepsilon}W_0|^{\frac{N+\alpha}{N}})\Big(\frac{N+\alpha}{N}\varepsilon_1^{-1}G(\varepsilon_2t_{\varepsilon}W_0)
+\varepsilon_1^{-1}\varepsilon_2g(\varepsilon_2t_{\varepsilon}W_0)t_{\varepsilon}W_0\Big)dx\\
&+\int_{\mathbb R^N}(I_{\alpha}*\varepsilon_1^{-1}G(\varepsilon_2t_{\varepsilon}W_0))\varepsilon_1^{-1}\varepsilon_2g(\varepsilon_2t_{\varepsilon}W_0)t_{\varepsilon}W_0dx.\\
\end{aligned}$$
If $t_{\varepsilon}\geq1$, by {\bf  (H4)}, we have
$$\begin{aligned}
\int_{\mathbb R^N}\varepsilon^{\sigma}|\nabla W_0|^2+|W_0|^2
&\geq t_{\varepsilon}^{\frac{2\alpha}{N}}\Big\{\frac{N+\alpha}{N}\int_{\mathbb R^N}(I_{\alpha}*|W_0|^{\frac{N+\alpha}{N}})|W_0|^{\frac{N+\alpha}{N}}\\
&+\frac{a(N+\alpha+Nq)}{Nq}\varepsilon^{\sigma}\int_{\mathbb R^N}(I_{\alpha}*|W_0|^{\frac{N+\alpha}{N}})|W_0|^{q}+o(\varepsilon^{\sigma})\Big\}.\\
\end{aligned}$$
Hence,
$$
1\le t_{\varepsilon}\le\left(\frac{\frac{2(N+\alpha)\tilde m_0}{\alpha}
+\varepsilon^{\sigma}\int_{\mathbb R^N}|\nabla W_0|^2}{\frac{2(N+\alpha)\tilde m_0}{\alpha}+\frac{a(N+\alpha+Nq)}{Nq}\varepsilon^{\sigma}\int_{\mathbb R^N}
(I_{\alpha}*|W_0|^{\frac{N+\alpha}{N}})|W_0|^{q}+o(\varepsilon^{\sigma})}\right)^{\frac{N}{2\alpha}}.
$$
If $t_{\varepsilon}\leq1$, by {\bf (H4)}, we have
$$\begin{aligned}
\int_{\mathbb R^N}\varepsilon^{\sigma}|\nabla W_0|^2+|W_0|^2
&\leq t_{\varepsilon}^{\frac{2\alpha}{N}}\Big\{\frac{N+\alpha}{N}\int_{\mathbb R^N}(I_{\alpha}*|W_0|^{\frac{N+\alpha}{N}})|W_0|^{\frac{N+\alpha}{N}}\\
&+\frac{a(N+\alpha+Nq)}{Nq}\varepsilon^{\sigma} \int_{\mathbb R^N}(I_{\alpha}*|W_0|^{\frac{N+\alpha}{N}})|W_0|^{q}+o(\varepsilon^{\sigma})\Big\}.\\
\end{aligned}$$
Hence,
$$
1\ge t_{\varepsilon}\ge \left(\frac{\frac{2(N+\alpha)\tilde m_0}{\alpha}
+\varepsilon^{\sigma}\int_{\mathbb R^N}|\nabla W_0|^2}{\frac{2(N+\alpha)\tilde m_0}{\alpha}+\frac{a(N+\alpha+Nq)}{Nq}\varepsilon^{\sigma} \int_{\mathbb R^N}
(I_{\alpha}*|W_0|^{\frac{N+\alpha}{N}})|W_0|^{q}+o(\varepsilon^{\sigma})}\right)^{\frac{N}{2\alpha}}.
$$
So we conclude that $t_{\varepsilon}\to1$ as $\varepsilon \to 0$.

Since $W_0=W_{\rho_0}$ and  $g_0(\rho_0)>0$, we have 
$$\frac{a}{q}\int_{\mathbb R^N}(I_{\alpha}*|W_0|^{\frac{N+\alpha}{N}})|W_0|^{q}>\frac{1}{2}\int_{\mathbb R^N}|\nabla W_0|^2,$$
therefore, it follows from \eqref{e421} that there exists a constant $C>0$ such that
$$\tilde m_{\varepsilon}\leq \tilde m_0-C\varepsilon^{\sigma}$$ for small $\varepsilon>0$,
and the conclusion follows.
\end{proof}

\begin{lemma}\label{lem45}
Under the conditions {\bf (H1)} and {\bf (H4)}, we have
$$\int_{\mathbb R^N}(I_{\alpha}*|w_{\varepsilon}|^{\frac{N+\alpha}{N}})|w_{\varepsilon}|^{q}\sim\|\nabla w_{\varepsilon}\|_2^2\sim1\ \ \  \mbox{as}\ \ \  \varepsilon\to0.$$
\end{lemma}
\begin{proof}
First, combining \eqref{e46}, \eqref{e414}, \eqref{e420}, {\bf (H4)} and Lemma \ref{lem25}, we have
$$\begin{aligned}
\tilde m_0&\leq \sup_{t\geq0}J_0((w_{\varepsilon})_t)\\
&\leq \tilde m_{\varepsilon}+
\frac{1}{2}\tau(w_\varepsilon)^{\frac{N+\alpha}{\alpha}}\int_{\mathbb R^N}(I_{\alpha}\ast (2|w_\varepsilon|^{\frac{N+\alpha}{N}}+\varepsilon_1^{-1}G(\varepsilon_2w_{\varepsilon}))\varepsilon_1^{-1}G(\varepsilon_2w_{\varepsilon})dx\\
& \quad-\frac{1}{2}\tau(w_\varepsilon)^{\frac{N-2}{\alpha}}\int_{\mathbb R^N}(I_{\alpha}*|w_{\varepsilon}|^{\frac{N+\alpha}{N}})
\left[\frac{N}{2}\varepsilon_1^{-1}g(\varepsilon_2w_{\varepsilon})\varepsilon_2w_{\varepsilon}-\frac{N+\alpha}{2}\varepsilon_1^{-1}G(\varepsilon_2w_{\varepsilon})\right]dx\\
& \quad-\frac{1}{2}\tau(w_\varepsilon)^{\frac{N-2}{\alpha}}\int_{\mathbb R^N}(I_{\alpha}*\varepsilon_1^{-1}G(\varepsilon_2w_{\varepsilon}))
\left[\frac{N}{2}\varepsilon_1^{-1}g(\varepsilon_2w_{\varepsilon})\varepsilon_2w_{\varepsilon}-\frac{N+\alpha}{2}\varepsilon_1^{-1}G(\varepsilon_2w_{\varepsilon})\right]dx\\
&\leq \tilde m_{\varepsilon}+\varepsilon^{\sigma}\left(\frac{a[N+\alpha+4-Nq]}{4q}\right)\int_{\mathbb R^N}(I_{\alpha}*|w_{\varepsilon}|^{\frac{N+\alpha}{N}})
|w_{\varepsilon}|^{q}+o(\varepsilon^{\sigma}),
\end{aligned}$$
and it follows from Lemma \ref{lem44} that
$$\int_{\mathbb R^N}(I_{\alpha}*|w_{\varepsilon}|^{\frac{N+\alpha}{N}})|w_{\varepsilon}|^{q}
\geq \frac{4q(\tilde m_0-\tilde m_{\varepsilon}+o(\varepsilon^\sigma))\varepsilon^{-\sigma}}{a[N+\alpha+4-Nq]}\geq C>0.
$$

On the other hand, by the Hardy-Littlewood-Sobolev inequality, H\"{o}lder inequality and the boundedness of $\{w_{\varepsilon}\}$ in $H^1(\mathbb R^N)$, we have
$\int_{\mathbb R^N}(I_{\alpha}*|w_{\varepsilon}|^{\frac{N+\alpha}{N}})|w_{\varepsilon}|^{q}\leq C$.
Therefore, $\int_{\mathbb R^N}(I_{\alpha}*|w_{\varepsilon}|^{\frac{N+\alpha}{N}})|w_{\varepsilon}|^{q}\sim 1$ as $\varepsilon\to 0$.

Finally, it follows from \eqref{e45} that
$$\|\nabla w_{\varepsilon}\|_2^2\sim1\ \ \mbox{as}\ \ \varepsilon\to 0.$$
The proof is complete.
\end{proof}

The following result is a special case of the classical Brezis-Lieb \cite{BLE1983} for Riesz potentials and for a proof, we refer the reader to \cite{MS2013}.

\begin{lemma}\label{lem46}
Let $N \in \mathbb N$, $\alpha \in (0,N)$ and $\{w_{n}\}$ be a bounded sequence in $L^2(\mathbb R^N)$. If $w_{n}\to w$ almost everywhere
on $\mathbb R^N$ as $n\to \infty$, then
$$\lim_{n\to\infty}\mathbb D(w_n)-\mathbb D(w_n-w_0)=\mathbb D(w_0).$$
\end{lemma}

\begin{lemma}\label{lem47}
Assume that $N\geq3$, {\bf (H1)} and {\bf (H4)} hold, then there exists $\xi_{\varepsilon}\in(0,\infty)$ satisfying
$$\xi_{\varepsilon}\thicksim \varepsilon^{-\frac{(2+\alpha)[Nq-(N+\alpha)]}{2\alpha(N+\alpha+4-Nq)}},$$
such that the rescaled ground states
$$w_{\varepsilon}(x)=\xi_{\varepsilon}^{\frac{N}{2}}v_{\varepsilon}(\xi_{\varepsilon}x)$$
converge to $W_{\rho_0}$ in $H^1(\mathbb R^N)$ as
$\varepsilon\to0$, where $W_{\rho_0}$ is a positive ground state of the equation \eqref{e19} with
\begin{equation}\label{e4222}
\rho_0=\left(\frac{2q\int_{\mathbb R^N}|\nabla W_1|^2}
{a[Nq-(N+\alpha)]\int_{\mathbb R^N}(I_{\alpha}*|W_1|^{\frac{N+\alpha}{N}})|W_1|^{q}}\right)^{\frac{2}{N+\alpha+4-Nq}}.
\end{equation}
\end{lemma}
\begin{proof}
The proof is similar to that in \cite[Lemma 4.8]{MM}. For the readers' convenience, we give the detail proof.  For any $\varepsilon_n\rightarrow0$ as $n\rightarrow\infty$, $w_{\varepsilon_n}$ is a positive radially symmetric function, and by Lemma \ref{lem42}, $\{w_{\varepsilon_n}\}$ is bounded in $H^1(\mathbb R^N)$,
therefore, there exists $w_{0} \in H^1(\mathbb R^N)$ satisfying $w=\frac{N+\alpha}{N}(I_{\alpha}*|w|^{\frac{N+\alpha}{N}})|w|^{\frac{N+\alpha}{N}-2}w$ such that up to a subsequence, there hold 
\begin{equation}\label{e4211}
w_{\varepsilon_n}\rightharpoonup w_{0}\ \ \mbox{in}\ \ H^1(\mathbb R^N),\ \ \ w_{\varepsilon_n}\rightarrow w_{0}\ \ \mbox{in}\ \ L^{s}(\mathbb R^N)\ \ \mbox{for} \ \ \forall\ s\in(2,2^*),
\end{equation}
and
\begin{equation}\label{e422}
w_{\varepsilon_n}\rightarrow w_{0}\ \ \mbox{in}\ \ L_{loc}^2(\mathbb R^N),\ \ \ w_{\varepsilon_n}(x)\rightarrow w_{0}(x)\ \ \mbox{a.e.  on} \ \ \mathbb R^N.
\end{equation}
It is not hard to show that
$$\begin{aligned}
J_{0}(w_{\varepsilon_n})&=J_{\varepsilon_n}(w_{\varepsilon_n})
+\frac{1}{2}\int_{\mathbb R^N}(I_\alpha\ast \varepsilon_1^{-1}G(\varepsilon_2w_{\varepsilon_n}))
\cdot\varepsilon_1^{-1}G(\varepsilon_2w_{\varepsilon_n})dx\\
&+\int_{\mathbb R^N}(I_\alpha\ast |w_{\varepsilon_n}|^{\frac{N+\alpha}{N}})\cdot \varepsilon_1^{-1}G(\varepsilon_2w_{\varepsilon_n})dx
-\frac{1}{2}\varepsilon_n^{\sigma}\int_{\mathbb R^N}|\nabla w_{\varepsilon_n}|^2\\
&=\tilde m_{\varepsilon_n}+o(1)=\tilde m_{0}+o(1),
\end{aligned}$$
and
$$\begin{aligned}
J'_{0}(w_{\varepsilon_n})w&=J'_{\varepsilon_n}(w_{\varepsilon_n})w
+\int_{\mathbb R^N}(I_\alpha\ast |w_{\varepsilon_n}|^{\frac{N+\alpha}{N}})
\varepsilon_1^{-1}\varepsilon_2g(\varepsilon_2w_{\varepsilon_n})wdx\\
&+\int_{\mathbb R^N}(I_\alpha\ast \varepsilon_1^{-1}G(\varepsilon_2 w_{\varepsilon_n}))(\frac{N+\alpha}{N}|w_{\varepsilon_n}|^{\frac{N+\alpha}{N}-2}w_{\varepsilon_n} 
+\varepsilon_1^{-1}\varepsilon_2g(\varepsilon_2w_{\varepsilon_n}))wdx\\
&-\varepsilon_n^{\sigma}\int_{\mathbb R^N}\nabla w_{\varepsilon_n}\nabla w=o(1).
\end{aligned}$$
Therefore, $\{w_{\varepsilon_n}\}$ is a PS sequence of $J_0$ at level
$\tilde m_0=\frac{\alpha}{2N}\left(\frac{N}{N+\alpha}S_{\alpha}\right)^{\frac{N+\alpha}{\alpha}}$.

By Lemma \ref{lem45}, arguing as the proof of \eqref{e316}, we obtain  $\int_{\mathbb R^N}(I_\alpha\ast |w_0|^{\frac{N+\alpha}{N}})|w_0|^q\not=0$. Therefore,  $w_0\neq 0$. Notice that $J_0'(w_0)=0$, we obtain $w_0\in \mathcal N_0$ and hence $\tilde m_0\leq J_0(w_0)$. By Lemma \ref{lem44} and Lemma \ref{lem46}, we have
$$\begin{aligned}
o(1)=&\tilde m_{\varepsilon_n}-\tilde m_0\ge J_{\varepsilon_n}(w_{\varepsilon_n})-J_0(w_0)\\
    \geq&\frac{1}{2}(\|w_{\varepsilon_n}\|_2^2-\|w_0\|_2^2)-\frac{1}{2}(\mathbb D(w_{\varepsilon_n})-\mathbb D(w_{0}))
    +\frac{1}{2}\varepsilon_n^{\sigma}\int_{\mathbb R^N}|\nabla w_{\varepsilon_n}|^2\\
    -&\int_{\mathbb R^N}(I_{\alpha}*|w_{\varepsilon_n}|^{\frac{N+\alpha}{N}})\varepsilon_1^{-1}G(\varepsilon_2w_{\varepsilon_n})dx
     -\frac{1}{2}\int_{\mathbb R^N}(I_{\alpha}*\varepsilon_1^{-1}G(\varepsilon_2w_{\varepsilon_n}))\varepsilon_1^{-1}G(\varepsilon_2w_{\varepsilon_n})dx\\
     =&\frac{1}{2}\|w_{\varepsilon_n}-w_0\|_2^2-\frac{1}{2}\mathbb D(w_{\varepsilon_n}-w_{0})+o(1),
\end{aligned}$$
and
$$\begin{aligned}
0&=\langle J'_{\varepsilon}(w_{\varepsilon_n})-J'_{0}(w_{0}),w_{\varepsilon_n}-w_{0}\rangle\\
 &=\varepsilon_n^{\sigma}\int_{\mathbb R^N}\nabla w_{\varepsilon_n}\nabla (w_{\varepsilon_n}-w_0)+\int_{\mathbb R^N}w_{\varepsilon_n}(w_{\varepsilon_n}-w_0)\\
 &-\int_{\mathbb R^N}(I_{\alpha}*(|w_{\varepsilon_n}|^{\frac{N+\alpha}{N}}+\varepsilon_1^{-1}G(\varepsilon_2w_{\varepsilon_n})))\\
 &\ \ \ \ \ \ \ \cdot(\frac{N+\alpha}{N}|w_{\varepsilon_n}|^{\frac{N+\alpha}{N}-2}w_{\varepsilon_n}(w_{\varepsilon_n}-w_0)
  +\varepsilon_1^{-1}\varepsilon_2g(\varepsilon_2w_{\varepsilon_n})(w_{\varepsilon_n}-w_0))dx\\
 &-\int_{\mathbb R^N}w_0(w_{\varepsilon_n}-w_0)
  +\int_{\mathbb R^N}(I_{\alpha}*|w_0|^{\frac{N+\alpha}{N}})\frac{N+\alpha}{N}|w_0|^{\frac{N+\alpha}{N}-2}w_0(w_{\varepsilon_n}-w_0)dx\\
 &=\|w_{\varepsilon_n}-w_0\|_2^2-\frac{N+\alpha}{N}\mathbb D(w_{\varepsilon_n}-w_{0})+o(1).
\end{aligned}$$
Hence, it follows that
$$\|w_{\varepsilon_n}-w_0\|_2^2\leq\mathbb D(w_{\varepsilon_n}-w_{0})+o(1)=\frac{N}{N+\alpha}\|w_{\varepsilon_n}-w_0\|_2^2+o(1),$$
and hence $$\|w_{\varepsilon_n}-w_0\|_2^2\to 0\ \ \ \mbox{as}\ \ \ n \to \infty.$$
By Hardy-Littlewood-Sobolev inequality and Lemma \ref{lem46}, it then follows that
$$\lim_{n\to \infty}\mathbb D(w_{\varepsilon_n})=\mathbb D(w_0).$$

On the other hand, by the boundedness of $\{w_{\varepsilon_n}\}$ in $H^1(\mathbb R^N)$, we have
$$\begin{aligned}
\tilde m_{\varepsilon_n}=J_{\varepsilon}(w_{\varepsilon_n})
&=\frac{1}{2}\int_{\mathbb R^N}\varepsilon_n^{\sigma}|\nabla w_{\varepsilon_n}|^2+|w_{\varepsilon_n}|^2\\
&-\frac{1}{2}\int_{\mathbb R^N}(I_{\alpha}*(|w_{\varepsilon_n}|^{\frac{N+\alpha}{N}}
+\varepsilon_1^{-1}G(\varepsilon_2w_{\varepsilon_n})))(|w_{\varepsilon_n}|^{\frac{N+\alpha}{N}}+\varepsilon_1^{-1}G(\varepsilon_2w_{\varepsilon_n}))dx\\
&= \frac{1}{2}\int_{\mathbb R^N}|w_{\varepsilon_n}|^2
-\frac{1}{2}\int_{\mathbb R^N}(I_{\alpha}*|w_{\varepsilon_n}|^{\frac{N+\alpha}{N}})|w_{\varepsilon_n}|^{\frac{N+\alpha}{N}}+o(1).
\end{aligned}$$
Letting $\varepsilon \to 0$, it then follows from Lemma \ref{lem44} that
$$\tilde m_0= \frac{1}{2}\int_{\mathbb R^N}|w_0|^2
-\frac{1}{2}\int_{\mathbb R^N}(I_{\alpha}*|w_0|^{\frac{N+\alpha}{N}})|w_0|^{\frac{N+\alpha}{N}}=J_0(w_0).$$
 Thus, $w_0=W_{\rho}$ for some $\rho \in (0,\infty)$.

Moreover, by \eqref{e45}, we obtain
$$\begin{aligned}
\|\nabla w_0\|_2^2\leq \lim_{\varepsilon_n\to0}\|\nabla w_{\varepsilon_n}\|_2^2
&=\lim_{\varepsilon_n\to0}\frac{a[Nq-(N+\alpha)]}{2q}
\int_{\mathbb R^N}(I_{\alpha}*|w_{\varepsilon_n}|^{\frac{N+\alpha}{N}})|w_{\varepsilon_n}|^{q}\\
&=\frac{a[Nq-(N+\alpha)]}{2q}
\int_{\mathbb R^N}(I_{\alpha}*|w_0|^{\frac{N+\alpha}{N}})|w_0|^{q},
\end{aligned}$$
from which it follows that
$$\rho\geq \left(\frac{2q\|\nabla W_1\|_2^2}{a[Nq-(N+\alpha)]
\int_{\mathbb R^N}(I_{\alpha}*|W_1|^{\frac{N+\alpha}{N}})|W_1|^{q}}\right)^{\frac{2}{N+\alpha+4-Nq}}.$$
If $\rho=\rho_0$, then \eqref{e45} implies that $\lim_{\varepsilon_n \to 0}\|\nabla w_{\varepsilon_n}\|_2^2=\|\nabla W_{\rho_0}\|_2^2$, and hence
$w_{\varepsilon_n}\to W_{\rho_0}$ in $\mathcal{D}^{1,2}(\mathbb R^N)$.

Let
$$M_{\varepsilon_n}=w_{\varepsilon_n}(0)\ \ \ \mbox{and}\ \ \ z_{\varepsilon_n}=M_{\varepsilon_n}[W_{\rho_0}(0)]^{-1},$$
where $\rho_0$ is given in \eqref{e421}. We further perform a scaling
$$\bar{w}_{\varepsilon_n}(x)=z_{\varepsilon_n}^{-1}w_{\varepsilon_n}(z_{\varepsilon_n}^{-\frac{2}{N}}x),$$
then $$\bar{w}_{\varepsilon_n}(0)=z_{\varepsilon_n}^{-1}w_{\varepsilon_n}(0)=M_{\varepsilon_n}^{-1}[W_{\rho_0}(0)]w_{\varepsilon_n}(0)=W_{\rho_0}(0),$$
and $\bar{w}_{\varepsilon_n}$ satisfies the rescaled equation
\begin{equation}\label{e423}
\begin{aligned}
-\varepsilon_n^{\sigma}z_{\varepsilon_n}^{\frac{4}{N}}\Delta \bar{w}_{\varepsilon_n}+\bar{w}_{\varepsilon_n}
=&\left(I_{\alpha}*(|\bar{w}_{\varepsilon_n}|^{\frac{N+\alpha}{N}} +\varepsilon_1^{-1}z_{\varepsilon_n}^{-\frac{N+\alpha}{N}}
G(\varepsilon_2z_{\varepsilon_n}\bar{w}_{\varepsilon_n}))\right)\\
\cdot&\left(\frac{N+\alpha}{N}|\bar{w}_{\varepsilon_n}|^{\frac{N+\alpha}{N}-2}\bar{w}_{\varepsilon_n}
+\varepsilon_1^{-1}\varepsilon_2z_{\varepsilon_n}^{-\frac{\alpha}{N}}g(\varepsilon_2z_{\varepsilon_n}\bar{w}_{\varepsilon_n})\right).
\end{aligned}
\end{equation}
The corresponding functional is given by
$$\begin{aligned}
\bar{J}_{\varepsilon_n}(\bar{w}_{\varepsilon_n})
&=\frac{1}{2}\int_{\mathbb R^N}\varepsilon_n^{\sigma}z_{\varepsilon_n}^{\frac{4}{N}}|\nabla \bar{w}_{\varepsilon_n}|^2+|\bar{w}_{\varepsilon_n}|^2\\
&-\frac{1}{2}\int_{\mathbb R^N}(I_{\alpha}*(|\bar{w}_{\varepsilon_n}|^{\frac{N+\alpha}{N}}
+\varepsilon_1^{-1}z_{\varepsilon_n}^{-\frac{N+\alpha}{N}}G(\varepsilon_2z_{\varepsilon_n}\bar{w}_{\varepsilon_n})))\\
&\ \ \ \ \ \ \ \ \ \ \cdot(|\bar{w}_{\varepsilon_n}|^{\frac{N+\alpha}{N}}
+\varepsilon_1^{-1}z_{\varepsilon_n}^{-\frac{N+\alpha}{N}}G(\varepsilon_2z_{\varepsilon_n}\bar{w}_{\varepsilon_n}))dx.
\end{aligned}$$
Moreover, we have
\begin{itemize}
\item[(i)] $\|\bar{w}_{\varepsilon_n}\|_2^2=\|w_{\varepsilon_n}\|_2^2,\ \ \ \ \ \
z_{\varepsilon_n}^{\frac{4}{N}}\|\nabla \bar{w}_{\varepsilon_n}\|_2^2=\|\nabla w_{\varepsilon_n}\|_2^2.$
\item[(ii)] $\int_{\mathbb R^N}(I_{\alpha}*|\bar{w}_{\varepsilon_n}|^{\frac{N+\alpha}{N}})|\bar{w}_{\varepsilon_n}|^{\frac{N+\alpha}{N}}
             =\int_{\mathbb R^N}(I_{\alpha}*|w_{\varepsilon_n}|^{\frac{N+\alpha}{N}})|w_{\varepsilon_n}|^{\frac{N+\alpha}{N}}.$
\end{itemize}

By \eqref{e422}, for any $\varepsilon_n\to 0$, there exists $\rho\geq\rho_0$ such that
$$M_{\varepsilon_n}=w_{\varepsilon_n}(0)\to W_{\rho}(0)=\rho^{-\frac{N}{2}}W_1(0)\leq \rho_0^{-\frac{N}{2}}W_1(0)<\infty,$$
which yields that $M_{\varepsilon_n}\leq C$ for some $C>0$ and any small $\varepsilon_n>0$. Suppose that there exists a sequence  of $\{\varepsilon_n\}$ (still denoted by $\varepsilon_n$)  satisfies $\varepsilon_n\to 0$ and  $M_{\varepsilon_n}\to 0$.
Then by \eqref{e422}, up to a subsequence, $M_{\varepsilon_n}=w_{\varepsilon_n}(0) \to W_{\rho}(0)\neq 0$ for some $\rho \in (0,\infty)$. This leads to a contradiction. Therefore, there exists some $c>0$ such that $M_{\varepsilon_n}\geq c>0$.

Let
$$\xi_{\varepsilon_n}=z_{\varepsilon_n}^{\frac{2}{N}}\varepsilon_n^{\frac{(2+\alpha)[Nq-(N+\alpha)]}{2\alpha(N+\alpha+4-Nq)}},$$
then
$$\xi_{\varepsilon_n}\thicksim \varepsilon_n^{\frac{(2+\alpha)[Nq-(N+\alpha)]}{2\alpha(N+\alpha+4-Nq)}},$$
and for small $\varepsilon_n>0$, the rescaled family of ground states
$$\bar{w}_{\varepsilon_n}(x)=\xi_{\varepsilon_n}^{-\frac{N}{2}}v_{\varepsilon_n}(\xi_{\varepsilon_n}^{-1}x)$$
satisfies
$$\|\nabla \bar{w}_{\varepsilon_n}\|_2^2\thicksim \int_{\mathbb R^N}
(I_{\alpha}*|\bar{w}_{\varepsilon_n}|^{\frac{N+\alpha}{N}})|\bar{w}_{\varepsilon_n}|^{\frac{N+\alpha}{N}}
\thicksim \|\bar{w}_{\varepsilon_n}\|_2^2\thicksim1,$$
and as $\varepsilon_n \to 0$, $\bar{w}_{\varepsilon_n}$ converges to the extremal function $W_{\rho_0}$ in $\mathcal{D}^{1,2}(\mathbb R^N)$.
Then by
$\bar{w}_{\varepsilon_n}\to W_{\rho_0}$ in $L^2(\mathbb R^N)$, we conclude that $\bar{w}_{\varepsilon_n}\to W_{\rho_0}$ in $H^1(\mathbb R^N)$.
\end{proof}

\begin{proof}[Proof of Theorem \ref{t11}]
Since $w_{\varepsilon}\in \mathcal{P}_{\varepsilon}$, it follows from the boundedness of $\{w_\varepsilon\}$ in $H^1(\mathbb R^N)$ that
$$\begin{aligned}
\tilde m_{\varepsilon}&=\frac{1}{2}\int_{\mathbb R^N}\varepsilon^{\sigma}|\nabla w_{\varepsilon}|^2+|w_{\varepsilon}|^2\\
&\ \ \ -\frac{1}{2}\int_{\mathbb R^N}
(I_\alpha\ast (|w_{\varepsilon}|^{\frac{N+\alpha}{N}}+\varepsilon_1^{-1}G(\varepsilon_2w_{\varepsilon})))(|w_{\varepsilon}|^{\frac{N+\alpha}{N}}
+\varepsilon_1^{-1}G(\varepsilon_2w_{\varepsilon}))dx\\
&=\frac{2+\alpha}{2(N+\alpha)}\varepsilon^\sigma \int_{\mathbb R^N}|\nabla w_\varepsilon|^2+\frac{\alpha}{2(N+\alpha)}\int_{\mathbb R^N}|w_\varepsilon|^2\\
&=\frac{\alpha}{2(N+\alpha)}\int_{\mathbb R^N}|w_{\varepsilon}|^2+O(\varepsilon^{\sigma}).
\end{aligned}$$
Similarly, we also have
$$\tilde m_{0}=\frac{\alpha}{2(N+\alpha)}\int_{\mathbb R^N}|W_1|^2
=\frac{\alpha}{2N}\int_{\mathbb R^N}(I_{\alpha}*|W_1|^{\frac{N+\alpha}{N}})|W_1|^{\frac{N+\alpha}{N}}.$$

According to Lemma \ref{lem44}, we get that
$$\int_{\mathbb R^N}|W_1|^2-\int_{\mathbb R^N}|w_{\varepsilon}|^2
=\frac{2(N+\alpha)}{\alpha}(\tilde m_0-\tilde m_{\varepsilon})+O(\varepsilon^{\sigma})=O(\varepsilon^{\sigma}).$$
Also, by  $\|W_1\|_2^2=\frac{N+\alpha}{N}\int_{\mathbb R^N}(I_{\alpha}*|W_1|^{\frac{N+\alpha}{N}})|W_1|^{\frac{N+\alpha}{N}}
=\left(\frac{N}{N+\alpha}\right)^{\frac{N}{\alpha}}S_{\alpha}^{\frac{N+\alpha}{\alpha}}$, we conclude that
$$\|w_{\varepsilon}\|_2^2=\left(\frac{N}{N+\alpha}\right)^{\frac{N}{\alpha}}S_{\alpha}^{\frac{N+\alpha}{\alpha}}+O(\varepsilon^{\sigma}).$$
Then by $w_{\varepsilon}\in \mathcal{N}_{\varepsilon}$, we obtain
$$\begin{aligned}
\int_{\mathbb R^N}(I_{\alpha}*|w_{\varepsilon}|^{\frac{N+\alpha}{N}})|w_{\varepsilon}|^{\frac{N+\alpha}{N}}
=\frac{N}{N+\alpha}\|w_{\varepsilon}\|_2^2+O(\varepsilon^{\sigma})=\left(\frac{N}{N+\alpha}S_{\alpha}\right)^{\frac{N+\alpha}{\alpha}}+O(\varepsilon^{\sigma}).
\end{aligned}$$
Finally, by \eqref{e41}, Lemma \ref{lem44}, Lemma \ref{lem45} and Lemma \ref{lem47}, we conclude the proof.
\end{proof}


\section{Final remarks} 

We  discuss briefly   the asymptotic behavior of ground states as $\varepsilon\to +\infty$.
To this end, we also make the following stronger assumption
\begin{itemize}
\smallskip

\item[\bf (H5)] there exist $b>0$ and $r\in (\frac{N+\alpha}{N}, \frac{N+\alpha}{N-2})$
such that
$
\lim_{s\to \infty}g(s)/|s|^{r-1}=b.
$
\end{itemize}
\smallskip
\smallskip

Arguing as in \cite[Theorem 2.1, Theorem 2.3]{MM2024},   we obtain the following result. 

\begin{theorem}\label{t51}  Assume  {\bf (H1)} and {\bf (H5)} hold. Then 
the problem \eqref{e12} admits a positive ground state $u_{\varepsilon} \in H^1(\mathbb R^N)$, which is radially symmetric and radially nonincreasing. Furthermore, the following statements hold true: as $\varepsilon \to \infty$,
$$u_{\varepsilon}(0)=\|u_\varepsilon\|_\infty\sim \varepsilon^{\frac{2+\alpha}{4(r-1)}},$$
$$\|u_{\varepsilon}\|_2^2=\varepsilon^{\frac{N+\alpha+2-Nr}{2(r-1)}}
\left\{\frac{[(N+\alpha)-(N-2)r]b^2}{2r^2}\left(\frac{r}{b^2}S_{r}\right)^{\frac{r}{r-1}}+o_\varepsilon(1)\right\},$$
$$\|\nabla u_{\varepsilon}\|_2^2=\varepsilon^{\frac{(N+\alpha)-(N-2)r}{2(r-1)}}
\left\{\frac{[Nr-(N+\alpha)]b^2}{2r^2}\left(\frac{r}{b^2}S_{r}\right)^{\frac{r}{r-1}}+o_\varepsilon(1)\right\}.$$
The least energy  of the ground state satisfies
$$I_{\varepsilon}(u_{\varepsilon})=\varepsilon^{\frac{(N+\alpha)-(N-2)r}{2(r-1)}}
\left[\frac{b^2(r-1)}{2r^2}\left(\frac{r}{b^2}S_{r}\right)^{\frac{r}{r-1}}+o_{\varepsilon}(1)\right],$$
where $$S_{r}=\inf_{w\in H^1(\mathbb R^N)\setminus \{0\}}\frac{\int_{\mathbb R^N}|\nabla w|^2+|w|^2}{(\int_{\mathbb R^N}(I_{\alpha}*|w|^{r})|w|^{r})^{\frac{1}{r}}}.$$
Moreover, as $\varepsilon \to \infty$, up a subsequence,  the rescaled ground states
\begin{equation}\label{4.1}
w_\varepsilon(x)=\varepsilon^{-\frac{2+\alpha}{4(r-1)}}u_\varepsilon(\varepsilon^{-\frac{1}{2}} x),
\end{equation}
converge in $H^1(\mathbb R^N)$ to a positive solution $w_{\infty}\in H^1(\mathbb R^N)$ of the equation
$$-\Delta w+ w=\frac{b^2}{r}(I_{\alpha}*|w|^{r})|w|^{r-2}w.$$
\end{theorem}


The $L^2$--mass of the ground--state
\begin{equation}\label{e52}
M(\varepsilon):=\|u_\varepsilon\|_2^2,
\end{equation}
plays a key role in the analysis of stability of the corresponding standing--wave solution of the time--dependent NLS.  More precisely, the importance of $M(\varepsilon)$ is for instance seen in the Grillakis-Shatah-Strauss theory \cite{Grillakis-1,Grillakis-2, Shatah-1, Weinstein-1} of stability for these solutions within the time-dependent Schr\"odinger equation. The latter says that the solution $u_\varepsilon$ is orbitally stable when $M'(\varepsilon) > 0$ and that it is unstable when $M'(\varepsilon) < 0$. Therefore the intervals where $M(\varepsilon)$ is increasing furnish stable solutions whereas those where $M(\varepsilon)$ is decreasing correspond to unstable solutions. 

Fig. 2 outlines the limits of $M(\varepsilon)$ as $\varepsilon\to 0$ and $\varepsilon\to\infty$ and reveals the variation of $M(\varepsilon)$
for small $\varepsilon> 0$ and large $\varepsilon > 0$ when $(q, r)$ belongs to different regions in the $(q, r)$ plane, as
described in Theorem \ref{t11} and Theorem \ref{t51}.

\begin{figure}[h]
	\centering 
\begin{tabular}{c|c|c}
          \hline
	  &  $q\in (\frac{N+\alpha}{N},  \frac{N+\alpha+4}{N})$ & $q\in (\frac{N+\alpha+4}{N}, \frac{N+\alpha}{N-2})$ \\
	  	\hline
		$r\in (\frac{N+\alpha+2}{N}, \frac{N+\alpha}{N-2})$ & $(M(0),M(\infty))=(0,0)$ & $(M(\varepsilon_q),M(\infty))=(M(\varepsilon_q),0)$\\
	\hline
	$r\in (\frac{N+\alpha}{N},  \frac{N+\alpha+2}{N})$ & $(M(0),M(\infty))=(0,+\infty)$ &$ (M(\varepsilon_q),M(\infty))=(M(\varepsilon_q),+\infty)$  \\
	\hline
\end{tabular}
\caption{The variation of $M(\varepsilon)$ for small and large $\varepsilon$, here  $M(0):=\lim_{\varepsilon\to 0}M(\varepsilon)$ and $M(\infty):=\lim_{\varepsilon\to \infty}M(\varepsilon)$}
\label{table1}
\end{figure}
Finally, we remark that we can find intervals $(0,\varepsilon_0)$ and  $(\varepsilon_\infty,\infty)$,  on which  $M(\varepsilon)$ is increasing or decreasing  according to the asymptotic behavior of $M(\varepsilon)$. See \cite[Lemma A.1]{MM}.

\smallskip
\smallskip 
\smallskip 
\smallskip

\noindent
\textbf{Data Availability} This research has no associated data.

\smallskip
\smallskip

\noindent \textbf{Conflict of interest}  The authors are not aware of any Conflict of interest.

\smallskip
\smallskip

\noindent \textbf{Ethics Approval} This research did not require ethics approval.


\end{document}